\documentclass[a4paper]{amsart}
\usepackage[utf8]{inputenc}
\usepackage[T1]{fontenc}

\usepackage{times}

\usepackage[pagebackref,colorlinks=true,pdfpagemode=none,urlcolor=blue,
linkcolor=blue,citecolor=blue]{hyperref}

\usepackage{amsmath,amsfonts,amssymb,amsthm,bm} 
\usepackage{amsthm, amssymb, amsmath, amsfonts, mathrsfs}

\usepackage{mathrsfs}
\usepackage{MnSymbol}
\usepackage{scalerel} 

\usepackage{color}
\usepackage{accents}

\usepackage{appendix}

\usepackage{bbm}

\definecolor{labelkey}{gray}{.8}
\definecolor{refkey}{gray}{.8}

\definecolor{darkred}{rgb}{0.9,0.1,0.1}
\definecolor{darkgreen}{rgb}{0,0.5,0}

\newcommand{\tnorm}[1]{{\left\vert\kern-0.25ex\left\vert\kern-0.25ex\left\vert #1 
    \right\vert\kern-0.25ex\right\vert\kern-0.25ex\right\vert}}

\newtheorem{theorem}{Theorem}[section]
\newtheorem{lemma}[theorem]{Lemma}

\newtheorem{corollary}[theorem]{Corollary}
\newtheorem{proposition}[theorem]{Proposition}

\theoremstyle{remark}
\newtheorem{remark}[theorem]{Remark}

\renewenvironment{proof}[1][Proof]{ {\itshape \noindent {#1.}} }{$\Box$
\medskip}

\numberwithin{equation}{section}
\newcommand{\R}{\mathbb{R}}

\newcommand{\e}{\ubar{e}}
\newcommand{\F}{\mathcal{F}}

\newcommand{\B}{\mathcal{B}}

\newcommand{\G}{\mathcal{G}}

\newcommand{\A}{\mathcal{A}}

\newcommand{\D}{\mathcal{D}}

\def\les{\lesssim}

\newcommand{\EE}{\mathbf{E}}

\newcommand{\Cov}{\mathrm{Cov}}

\renewcommand{\d}{\mathrm{d}}
\renewcommand{\e}{\mathrm{e}}

\begin{document}
\title{spatial decorrelation of KPZ from narrow wedge}
\author{Yu Gu \and Fei Pu}

\address[Yu Gu]{Department of Mathematics, University of Maryland, College Park, MD 20742, USA}

\address[Fei Pu]{Laboratory of Mathematics and Complex Systems,
School of Mathematical Sciences, Beijing Normal University, 100875, Beijing, China.}

\begin{abstract}
	We  study the spatial 
    decorrelation of the solution to the KPZ equation with narrow wedge initial data. For fixed $t>0$, we determine the  decay rate of the spatial covariance function, showing that $\Cov[h(t,x),h(t,0)]\sim \frac{t}{x}$ as 
    $x\to\infty$. In addition, we prove that the finite-dimensional distributions of the properly rescaled spatial average of the height function converge to those of a Brownian motion. 
  
    \bigskip
    
\end{abstract}
 
\maketitle

\section{Introduction}

\subsection{Motivation and main result}
The KPZ equation, introduced by Kardar, Parisi and Zhang \cite{KPZ86} in their study of fluctuations in growing interfaces subject to random perturbations, is a non-linear stochastic partial differential equation formally written as
\begin{align}\label{KPZ}
\partial_t h(t,x)= \tfrac12\partial_x^2h(t,x) +  \tfrac{1}{2}(\partial_xh(t,x))^2 + \xi(t,x), \quad t>0, x\in \R,
\end{align}
where $\xi$ is a space-time white noise. The Cole-Hopf solution to the KPZ equation \eqref{KPZ} is defined as
\begin{align*}
h(t,x)= \log Z(t,x),
\end{align*}
where $Z$ is the solution to the stochastic heat equation (SHE) 
\begin{align}\label{eq:SHE}
\partial_t Z(t,x)= \tfrac12\partial_x^2Z(t,x)  + Z(t,x)\xi(t,x).
\end{align}
Throughout the paper, we work with the narrow wedge initial data for KPZ, namely, $Z(0,x)=\delta_0(x)$, 
where $\delta_0$ is the Dirac measure at $0$. 

Our goal  is to study the correlation structure of the spatial process $\{h(t,x)\}_{x\in\R}$ for fixed $t>0$. In particular, we are interested in the decay rate of $\Cov[h(t,x),h(t,0)]$ as $x\to\infty$. The question is motivated by the following two observations.

First, for the Airy$_1$ and Airy$_2$ processes, which are the scaling limits of the KPZ equation with flat and narrow wedge initial conditions respectively, the decay rates of the covariance are well understood. Indeed,  Pr\"ahofer-Spohn \cite{ps}   {indicated} that 
\begin{align}\label{e.decayA2}
\Cov[\A_2(x), \A_2(0)]\sim x^{-2}, \quad \text{as $x\to\infty$},
\end{align}
and Basu et al. \cite{BBF23} have recently derived that
\begin{align}\label{e.decayA1}
\Cov[\A_1(x), \A_1(0)]= \e^{-\frac{4}{3}(1+o(1))x^3}, \quad \text{as $x\to\infty$}.
\end{align}
 Here and throughout, $a\sim b$ means $\tfrac{a}{b}\to1$ in the relevant limit.  Interpreting $\A_1(x)$ and $\A_2(x)$  as describing fluctuations of the free energy of a continuum directed random polymer (CDRP) \cite{CDRP} starting from $x$, with the endpoint either moving freely (Airy$_1$) or fixed at the origin (Airy$_2$), one sees that the difference in decorrelation rates between \eqref{e.decayA2} and \eqref{e.decayA1} arises from path interactions near the terminal point. In the Airy$_2$ case, the two polymer paths—starting from $x$ and $0$—are both conditioned to end at the origin and hence experience a shared random environment near the endpoint, resulting in stronger correlations. Viewing the KPZ equation as a   {pre-limiting} model, it is known that, under the  flat initial condition, $\Cov[h(t,x), h(t,0)]$ decays exponentially as $x\to\infty$, as shown via Poincar\'e inequality and the estimate on the Malliavin derivative of the solution to the stochastic heat equation (see Chen et al. \cite{CKNP23} for an upper bound). It is thus natural to ask whether a result analogous to \eqref{e.decayA2} holds for the KPZ equation with narrow wedge initial data, and whether there is a meaningful connection to the behavior described in \eqref{e.decayA2}.

Second, spatial decorrelation plays a central role in the studies of spatial averages of SPDEs and the associated central limit theorem. When the field decorrelates sufficiently fast, one naturally expects a Gaussian limit; in contrast, slowly decorrelating fields may give rise to non-Gaussian behavior. As previously noted, narrow wedge initial data induces strong correlations   through the influence of the noise near the origin. Although   \cite{pufei} established that the spatial averages of the Airy$_1$ and Airy$_2$ processes both exhibit Gaussian fluctuations, it is natural to ask whether similar Gaussian behavior persists for the KPZ equation itself, given the strong correlations induced by the narrow wedge initial condition.


Before presenting the main result, we first observe, due to the shear invariance of the white noise   {(see Proposition 2.3(iii) of \cite{chris})},   $h(t,x)-\log p_t(x)=\log \tfrac{Z(t,x)}{p_t(x)}$ is stationary in the $x-$variable, where $p_t(x)=\tfrac{1}{\sqrt{2\pi}}\exp(-\tfrac{x^2}{2t})$ is the standard heat kernel. It follows that, for each $t>0$ fixed, $h(t,x)-\EE[ h(t,x)]$ is stationary in $x$. We are now ready to state our main result:
\begin{theorem}\label{th:cov}
       Fix any $t>0$. As $x\to\infty$, 
       \begin{align}\label{e.tx}
        \Cov[h(t,x), h(t,0)]\sim \frac{t}{x}.
       \end{align}
\end{theorem}


The above result shows that the spatial covariance function of the KPZ equation with narrow wedge initial data is \emph{not} integrable,   {with a \emph{marginal} decay of $1/x$. The main technical aspect of the paper is on determining the precise prefactor $t$. } As a result of the slow decay of covariance, the spatial average of the solution cannot be normalized by the standard central limit scaling. In fact, we have the following theorem.
\begin{theorem}\label{th:CLT}
        As $N\to\infty$, 
       \begin{align*}
        \left\{\frac{1}{\sqrt{N\log N}}\int_0^N\left(h(t,x)-\EE[h(t,x)]\right)\d x\right\}_{t\in \R_+}\xrightarrow{\rm fdd} \sqrt{2}{\rm B},
       \end{align*}
       where $\rm B$ denotes the standard Brownian motion, and the symbol $\xrightarrow{\rm fdd}$ refers to convergence of finite-dimensional distributions.
\end{theorem}
The $\log N$ factor in the above normalization, which also appears in the CLT for spatial average of SHE with Dirac initial data (see \cite[Theorem 1.3]{CKNPa}), 
is due to the $\frac1x$-decay rate in Theorem \ref{th:cov}. Moreover, the limiting Gaussian process, $\sqrt{2}{\rm B}$, coincides with that  in \cite[Theorem 1.3]{CKNPa}.
As we will see in the proof of Theorem \ref{th:CLT}, somewhat unexpectedly, the first Wiener chaos  of the spatial average dominates and gives rise to the limiting Gaussian behavior described in Theorem \ref{th:CLT}.

\subsection{Context and discussion}

The study of SHE and KPZ  has seen tremendous progress during the past years, and a comprehensive review of the literature is beyond the scope of this paper. The most physically relevant regime involves  simultaneously rescaling time and space so as to observe nontrivial fluctuations. In this regard, we refer to the reviews \cite{icreview,jqreview,QS15}, the recent developments \cite{landscape,MQR21,Vir20,wx} and the references cited there.

 
 The focus of this paper lies in a different direction and may be viewed as a continuation of a series of works by the second author and collaborators \cite{CKNPa,CKNP23,CKNP22,CKNP21}. It is also closely connected to a broader body of recent research on spatial averages and central limit theorems for stochastic heat- and wave-type equations; see, for instance, \cite{BNZ21,DNZ20, HNV2018, HNVZ20, NXZ22} and references therein.

There are relatively few tools available for the direct analysis of the KPZ equation. A key ingredient in our analysis is the use of the Clark–Ocone formula, which, in some sense, serves as an analogue of the mild formulation for the SHE. It allows one to express the centered height function as an It\^o integral, where the integrand involves a conditional expectation of the midpoint density of the continuum directed polymer. This reformulation reduces the problem of estimating $\Cov[h(t,x), h(t,0)]$ to analyzing a conditioned version of the overlap between two polymer paths connecting $(t,x)\to (0,0)$ and $(t,0)\to (0,0)$ respectively. In the case of narrow wedge initial data, this analysis turns out to be surprisingly delicate, as the dominant contribution to the covariance  $\Cov[h(t,x), h(t,0)]$ for large $x$ arises from the effect of the noise near the origin.

One may ask if the $\tfrac{1}{x}$ decay obtained in Theorem~\ref{th:cov} is  related in any way to the $\tfrac{1}{x^2}$ decay in  \eqref{e.decayA2}, or if the $\tfrac{t}{x}$ factor on the right-hand side of \eqref{e.tx} reflects the 1:2:3 scaling.  The answer is likely no -- or at least not directly. The two limits of $x\to\infty$ and $t\to\infty$ do not commute: in light of the weak convergence of $2^{1/3}t^{-1/3}(h(t,t^{2/3}\cdot )+\tfrac{t}{24})$ to the $\A_2(\cdot)$ as $t\to\infty$ (see \cite[Theorem 1.7]{wx}),  interchanging the limits $x\to\infty$ and $t\to\infty$ leads to inconsistent conclusion:
%
\begin{align*}
\frac{2^{2/3}}{t^{1/3}x} 
\stackrel{\infty \leftarrow x}{\sim} 
&\Cov\left[2^{1/3}t^{-1/3}h(t,t^{2/3}x),\, 2^{1/3}t^{-1/3}h(t,0)\right] \\
&\stackrel{t \to \infty}{\longrightarrow} 
\Cov[\mathcal{A}_2(x), \mathcal{A}_2(0)] 
\stackrel{x \to \infty}{\sim} 
\frac{1}{x^2}.
\end{align*}
On the other hand, one could also try to use \eqref{e.4162} to study the decorrelation of the KPZ solution in the joint limit where both time and space tend to infinity, by carefully analyzing the conditioned overlap of two paths from the CDRP. It would be interesting to establish a direct connection between the decorrelation of the Airy process and the geometric overlap of polymer paths.

  {After the paper was submitted, a similar strategy was implemented for the KPZ equation with flat initial data in \cite{CJ26}, where precise asymptotic decay of covariance function was derived.}

\bigskip

\subsection*{Organization of the paper.} The rest of paper is organized as follows. In Section~\ref{s.p}, we review the basic properties of SHE, including its Green’s function, Malliavin derivative, and related tools. We prove Theorems~\ref{th:cov} and \ref{th:CLT} in Sections~\ref{s.decaycor} and \ref{s.clt} respectively. Some technical lemmas are left in Section~\ref{s.tech}.


\section{Preliminary}
\label{s.p}

This section sets up the framework for the analysis of SHE. Let $\mathcal{H}= L^2(\R_+ \times \R)$.
The Gaussian family $\{ \xi(\phi)\}_{\phi \in \mathcal{H}}$ formed by the Wiener integrals
\[
	\xi(h)= \int_{\R_+}\int_\R  \phi(s,y)\, \xi(\d s\, \d y)
\]
defines an  {\it isonormal Gaussian process} on the Hilbert space $\mathcal{H}$.
In this framework we apply the Malliavin calculus (see Nualart \cite{Nualart}).
Denote by $D$ the derivative operator.
Let $\{\mathcal{F}_s\}_{s \geq 0}$ denote the filtration generated by
the space-time white noise $\xi$. The solution to the stochastic heat equation \eqref{eq:SHE} with the Dirac initial data satisfies
\begin{align*}
Z(t,x)= p_t(x)+ \int_0^t\int_\R p_{t-s}(x-y)Z(s,y)\,\xi(\d s\, \d y),
\end{align*}
where $p_t(x)=(2\pi t)^{-1/2}\e^{-x^2/(2t)}$. Moreover, the Malliavin derivative of $Z$
satisfies the following stochastic integral equation: for $s\in(0\,,t)$ and $y\in \R$, 
\begin{equation}\label{e.41711}
	D_{s, y}Z(t,x) = p_{t - s}(x - y)Z(s,y) + \int_s^t\int_\R
	p_{t - r}(x - z)D_{s, y}Z(r,z)\,\xi({\d r\, \d z}),
\end{equation}
(see Chen et al. \cite[Proposition 5.1]{CHN21}). 

  {Define $\{\G(t,x;s,y):0<s<t, x,y\in \R\}$ as the solution to the family of stochastic partial differential equations
\begin{align*}
\begin{cases}
\partial_t\G(t,x;s,y)=\frac12\partial_x^2\G(t,x;s,y) + \G(t,x;s,y)\xi(t,x),\\
\G(s,x;s,y)=\delta_0(x-y).
\end{cases}
\end{align*}
Equivalently, $\G$ is the Green's function of the SHE; see \cite{chris}. One  observes that 
\begin{align}\label{e.defG}
\G(t,x;s,y)= \frac{D_{s,y}Z(t,x)}{Z(s,y)}, \quad \text{for $s<t$ and $x,y\in \R$},
\end{align}
since both the left-hand side and the right-hand side of \eqref{e.defG} satisfy the 
\begin{align*}
\G(t,x;s,y)=p_{t-s}(x-y) + \int_s^t\int_\R
	p_{t - r}(x - z)\G(r,z; s,y)\,\xi({\d r\, \d z}),
\end{align*}
which can be seen from \eqref{e.41711}.}
 Furthermore, introduce the scaled Green' function
\begin{align*}
\bar{\G}(t,x;s,y)= \frac{\G(t,x;s,y)}{p_{t-s}(x-y)} \quad \text{for $s< t$ and $x,y\in \R$}.
\end{align*}
By the translation invariance of the white noise (see Proposition 2.3 of \cite{chris}), $\bar\G(t,x;s,y)$ has the same distribution as
$\bar\G(t-s;x-y;0,0)$ and for any $T>0$ and $k\in\R$
\begin{align}\label{eq:moments}
\sup_{\substack{0<s<t\leq T\\ x,y\in \R}}\EE[\bar\G(t,x;s,y)^k]<\infty,
\end{align}
(see \cite[Lemma 3.2]{chris}).

The following property of the Green's function of  the SHE will be used later in the proof of Theorems \ref{th:cov} and \ref{th:CLT}. 

\begin{lemma}\label{lem:shift}
      For $s<t$ and $x,y\in \R$,
     \begin{align}\label{eq:center}
      \EE\left[\frac{\bar{\G}(t,x;s,y)}{\bar{\G}(t,x;0,0)}\right]=  \EE\left[\frac{\bar{\G}(t,0;s,0)}{\int_\R p_{s(t-s)/t}(z+y-\frac{s}{t}x)\bar{\G}(t,0;s,z)\bar{\G}(s,z+y;0, 0)\d z}\right].
      \end{align}
\end{lemma}
\begin{proof}
       By the convolution formula for the Green's function of SHE (see \cite[Theorem 2.6]{chris}), we write   {
       \begin{align*}
        \frac{\bar{\G}(t,x;s,y)}{\bar{\G}(t,x;0,0)}  &= 
        \frac{p_t(x)\bar{\G}(t,x;s,y)}{\int_\R\G(t,x;s,y')\G(s,y';0, 0)\d y'} \\
                &=        \frac{\bar{\G}(t,x;s,y)}{\int_\R p_{s(t-s)/t}(y'-\frac{s}{t}x)\bar{\G}(t,x;s,y')\bar{\G}(s,y';0, 0)\d y'} ,
        \end{align*}}
               where we used the following identity on heat kernel
               \begin{equation}\label{PPPP}
	\frac{p_{t-s}(a)p_s(b)}{p_t(a+b)} =
	p_{s(t-s)/t}\left( b - \frac st (a+b)\right)
	\quad\text{for all $0<s<t$ and $a,b\in\R$}.
\end{equation}
Introduce the notation $\EE_{a,b}$ as the expectation with respect to the space-time white noise from time $a$ to $b$, recall the fact that
               $\bar{\G}(t,x;s,y)$ is independent of $\mathcal{F}_s$ for all $x,y \in \R$. For fixed $z\in \R$, the random variable \[
               \frac{\bar{\G}(t,x+z;s,y+z)}{\int_\R p_{s(t-s)/t}(y'-\frac{s}{t}x)\bar{\G}(t,x+z;s,y'+z)\bar{\G}(s,y';0, 0)\d y'}\stackrel{\text{law}}{=}\frac{\bar{\G}(t,x;s,y)}{\int_\R p_{s(t-s)/t}(y'-\frac{s}{t}x)\bar{\G}(t,x;s,y')\bar{\G}(s,y';0, 0)\d y'}
               \] under the expectation
        $\EE_{s,t}$ (see \cite[Proposition 2.3(i)]{chris}), 
      which implies 
               \begin{align*}
                       \EE\left[\frac{\bar{\G}(t,x;s,y)}{\bar{\G}(t,x;0,0)}\right]
                       &=       \EE_{0,s}\left[ \EE_{s,t}\left[\frac{\bar{\G}(t,0;s,y-x)}{\int_\R p_{s(t-s)/t}(y'-\frac{s}{t}x) \bar{\G}(t,0;s,y'-x)\bar{\G}(s,y';0, 0)\d y'}\right] \right]\\
                &=       \EE_{0,s}\left[ \EE_{s,t}\left[\frac{\bar{\G}(t,0;s,y-x)}{\int_\R p_{s(t-s)/t}(y'-\frac{s}{t}x) \bar{\G}(t,0;s,y'-x)\bar{\G}(s,y'-x;0, -x)\d y'}\right] \right]\\
                       &=       \EE_{0,s}\left[ \EE_{s,t}\left[\frac{\bar{\G}(t,0;s,y-x)}{\int_\R p_{s(t-s)/t}(y'+x-\frac{s}{t}x) \bar{\G}(t,0;s,y')\bar{\G}(s,y';0, -x)\d y'}\right] \right],
                       \end{align*}
    where the second equality holds by the stationarity of the process
    $\{\bar{\G}(s,y'+z;0,z):z\in \R\}$ under the expectation $\EE_{0, s}$.   {For the last expression, we use the shear invariance of the white noise which implies the following distributional identity 
    \begin{equation}\label{e.4231}
    \{\bar{\G}(t,0;s,y-x+z)\}_{z\in\R} \stackrel{\text{law}}{=}\{\bar{\G}(t,0;s,z)\}_{z\in\R},
    \end{equation} }
    from which we conclude that 
       \begin{align*}
        \EE\left[\frac{\bar{\G}(t,x;s,y)}{\bar{\G}(t,x;0,0)}\right] 
        &=\EE_{0, s}\left[\EE_{s,t}\left[\frac{\bar{\G}(t,0;s,0)}{\int_\R p_{s(t-s)/t}(y'+x-\frac{s}{t}x)\bar{\G}(t,0;s,y'+x-y)\bar{\G}(s,y';0, -x)\d y'}\right]\right]\\
                &=\EE_{0, s}\left[\EE_{s,t}\left[\frac{\bar{\G}(t,0;s,0)}{\int_\R p_{s(t-s)/t}(z-\frac{s}{t}x)\bar{\G}(t,0;s,z-y)\bar{\G}(s,z;0, 0)\d z}\right]\right]\\
           &=\EE\left[\frac{\bar{\G}(t,0;s,0)}{\int_\R p_{s(t-s)/t}(z-\frac{s}{t}x)\bar{\G}(t,0;s,z-y)\bar{\G}(s,z;0, 0)\d z}\right],
       \end{align*}
         {where in the second equality, we used stationarity to replace $\bar{\G}(s,y';0, -x)$ by $\bar{\G}(s,y'+x;0,0)$ and the change of variable $z=y'+x$.} The proof is complete.
\end{proof}
  {
\begin{remark}
       Lemma \ref{lem:shift} is a re-centering identity for the Green's
       function $\bar{\G}$. The endpoint $x$ on the left-hand side of 
       \eqref{eq:center} is moved to $0$ on the right-hand side of \eqref{eq:center}, at the cost of introducing a heat kernel in 
       the denominator. This is useful because the dependence $x$ is then
       confined to a deterministic heat kernel and that makes the large-$x$ analysis in the proofs of Theorems \ref{th:cov} and \ref{th:CLT} tractable.
\end{remark}}
We conclude this section with a brief overview of the notation used throughout the paper.
For every $Y\in L^k(\Omega)$ with $k\in[1,\infty)$, we write $\|Y\|_k=(\EE[|Y|^k])^{1/k}$.  Throughout we write ``$g_1(x)\lesssim g_2(x)$ for all $x\in X$'' when
there exists a real number $L$ such that $g_1(x)\le Lg_2(x)$ for all $x\in X$.
Alternatively, we might write ``$g_2(x)\gtrsim g_1(x)$ for all $x\in X$.'' By
``$g_1(x)\asymp g_2(x)$ for all $x\in X$'' we mean that $g_1(x)\lesssim g_2(x)$
 and $g_2(x)\lesssim g_1(x)$ for all $x\in X$. 
We also use $\hat{f}$ to denote the Fourier transform of $f$,
normalized so that
\[
	\hat{f}(x) = \int_{-\infty}^\infty
	\e^{ixy}f(y)\,\d y\qquad\text{for all $x\in\R$ and $f\in L^1(\R)$.}
\]
As already used in the proof of Lemma~\ref{lem:shift}, $\EE_{s,t}$ is the expectation on the noise during the time interval $[s,t]$.

\section{Spatial decorrelation: proof of Theorem \ref{th:cov}}
\label{s.decaycor}
In this section, we analyze the asymptotic behavior of the spatial covariance. The starting point is the Clark-Ocone formula (\cite[Proposition 6.3]{CKNP21}): 
    \begin{align*}
     h(t,x)&=\EE[h(t,x)]+ \int_0^t\int_\R \EE[D_{s,y}h(t,x)|\F_s]\,\xi(\d s\, \d y)\\
     &=\EE[h(t,x)]+ \int_0^t\int_\R \EE\left[\frac{D_{s,y}Z(t,x)}{Z(t,x)}|\F_s\right]\,\xi(\d s\, \d y)\\
     &=\EE[h(t,x)]+ \int_0^t\int_\R \A(t,x;s,y)p_{s(t-s)/t}(y-\frac{s}{t}x)\,\xi(\d s\, \d y),
    \end{align*}
    where $\A(t,x;s,y)$ is defined as
    \begin{align}\label{eq:A}
\A(t,x;s,y)=\EE\left[\frac{\bar{\G}(t,x;s,y)\bar\G(s,y;0,0)}{\bar{\G}(t,x;0,0)}|\F_s\right], \quad\quad s<t,x,y\in\R,
\end{align}
and we used \eqref{e.defG} in the calculation of the Malliavin derivative.

By Ito's isometry, the covariance function can be expressed as
    \begin{equation}\label{e.4162}
    \begin{aligned}
    &\Cov[h(t,x),h(t,0)]\\
    &\qquad\qquad= \int_0^{t}\int_\R \EE [\mathcal{X}(t,s,x,y)]p_{s(t-s)/t}(y)p_{s(t-s)/t}(y-\frac{s}{t}x) \d y\d s 
    \end{aligned}
    \end{equation}
    with 
    \[
    \mathcal{X}(t,s,x,y)= \mathcal{A}(t,x;s,y)  \mathcal{A}(t,0;s,y).
    \]
    
    As discussed earlier, the main contribution to the covariance comes from the noise near the origin. From the above expression, this can be seen  as follows: by the moment estimates on $\bar{\G}$ given by \eqref{eq:moments}, one can show that $\mathcal{X}(t,s,x,y)$ is bounded from above and below by positive constants   depending only on $t$. Consequently, the integral in \eqref{e.4162} is of the same order as 
\begin{equation}\label{e.6261}
\begin{aligned}
&\int_0^{t}\int_\R  p_{s(t-s)/t}(y)p_{s(t-s)/t}(y-\frac{s}{t}x) \d y\d s =\int_0^t p_{2s(t-s)/t}(\frac{s}{t}x) \d s\\
\end{aligned}
\end{equation}
For those $s$ away from $0$, the Gaussian kernel $p_{2s(t-s)/t}(\frac{s}{t}x)$ leads to a much faster decay  than $\tfrac{1}{x}$, thus, the main contribution to the above integral comes from those small $s$. Therefore, one needs to understand the behavior of   {$\mathcal{X}(t,s,x,y)$} for $s\ll1$. 

As a matter of fact, one can study the integral on the right-hand side of \eqref{e.6261} and derive that as $x\to\infty$ it is of order $x^{-1}$. This gives an upper and lower bounds of $\Cov[h(t,x),h(t,0)]$ by   {constants times} $x^{-1}$. To derive the precise decaying rate,  the main technical difficulty in the proof lies in dealing with the extra term $\EE [\mathcal{X}(t,s,x,y)]$ in the integrand of \eqref{e.4162}.

\subsection{A  technical estimate}

From the expression of $\A$ given in \eqref{eq:A}, it is straightforward to guess what the limit of $\mathcal{A}(t,x;s,y)$ is as $s\to0$: we replace the conditional expectation by the full expectation and pass to the limit of each  $\bar{\G}$ factor. This motivates us to define
\begin{align}\label{eq:GT}
g_t(x,y)=\EE\left[\frac{\bar{\G}(t,x;0,y)}{\bar{\G}(t,x;0,0)}\right].
\end{align}
We have the following key technical proposition:
\begin{proposition}\label{p.key}
Fix $t>0$ and $k\geq2$, there exists $C>0$ depending on $t$ and $k$ such that for all $s\in (0,t/4)$, 
\begin{equation}\label{e.errestimate1}
\|\mathcal{A}(t,x;s,y)-g_t(x,y)\|_k\leq C (s^{1/4}+d(xs^{1/4},ys^{1/4})+d(xs^{1/4},0)),
\end{equation}
with $d(x,y):=  {|x-y|\wedge 1}$.
\end{proposition}
\begin{proof}
    We write 
    \[
    \begin{aligned}
    \mathcal{A}(t,x;s,y)
    &=\EE_{s,t}\left[\frac{\bar{\G}(t,x;s,y)}{\bar{\G}(t,x;0,0)}\right]\bar{\G}(s,y;0,0). 
    \end{aligned}
    \]
    For the denominator, using the convolution formula for the Green's function, we rewrite it as 
    \[
    \begin{aligned}
    \bar{\G}(t,x;0,0)&=p_t(x)^{-1}\int_\R \G(t,x;s,y')\G(s,y';0,0)\d y'\\
    &=p_t(x)^{-1}\int_\R \G(t,x;s,y')p_s(y') \bar{\G}(s,y';0,0)\d y'\\
    &=p_t(x)^{-1}\int_\R \G(t,x;s,y')p_s(y')\d y'\\
    &\quad +p_t(x)^{-1} \int_\R \G(t,x;s,y')p_s(y') (\bar{\G}(s,y';0,0)-1)\d y'\\
    &=:I_1(t,s,x)+I_2(t,s,x).
    \end{aligned}
    \]
    So we have
    \begin{equation}\label{e.4171}
    \begin{aligned}
    \EE_{s,t}\left[\frac{\bar{\G}(t,x;s,y)}{\bar{\G}(t,x;0,0)}\right]&=\EE_{s,t}\left[\frac{\bar{\G}(t,x;s,y)}{I_1(t,s,x)+I_2(t,s,x)}\right]\\
    &=\EE_{s,t}\left[\frac{\bar{\G}(t,x;s,y)}{I_1(t,s,x) }\right]+\EE_{s,t}\left[\frac{\bar{\G}(t,x;s,y)(-I_2(t,s,x))}{\bar{\G}(t,x;0,0)I_1(t,s,x)} \right].
    \end{aligned}
    \end{equation}
    Thus, it follows that
\[
\begin{aligned}
&\mathcal{A}(t,x;s,y)-g_t(x,y)\\
&=\left(\EE_{s,t}\left[\frac{\bar{\G}(t,x;s,y)}{I_1(t,s,x) }\right]+\EE_{s,t} \left[\frac{\bar{\G}(t,x;s,y)(-I_2(t,s,x))}{\bar{\G}(t,x;0,0)I_1(t,s,x)}\right]\right) \bar{\G}(s,y;0,0) -g_t(x,y)\\
&=\left(\EE_{s,t}\left[\frac{\bar{\G}(t,x;s,y)}{I_1(t,s,x) }\right]+\EE_{s,t} \left[\frac{\bar{\G}(t,x;s,y)(-I_2(t,s,x))}{\bar{\G}(t,x;0,0)I_1(t,s,x)}\right]\right) [\bar{\G}(s,y;0,0)-1]\\
&\quad+\left(\EE \left[\frac{\bar{\G}(t,x;s,y)}{I_1(t,s,x) }\right]-g_t(x,y)\right)+\EE_{s,t} \left[\frac{\bar{\G}(t,x;s,y)(-I_2(t,s,x))}{\bar{\G}(t,x;0,0)I_1(t,s,x)}\right]\\
&=:err_1+err_2+err_3.
\end{aligned}
\] 

Note that  
\[
\begin{aligned}
\|I_2(t,s,x)\|_k\leq &p_t(x)^{-1}\int_\R \|\G(t,x;s,y')(\bar{\G}(s,y';0,0)-1)\|_k p_s(y') \d y'\\
\leq&  p_t(x)^{-1}\int_\R \|\G(t,x;s,y')\|_{2k}\|\bar{\G}(s,0;0,0)-1\|_{2k} p_s(y')\d y' .
\end{aligned}
\]
We use Burkholder's inequality and follow the same calculation in \cite[Lemma A.4]{CKNPa} to derive that for $T>0$ and $k\geq2$, there exists $C_{T,k}>0$ such that for all $0<s\leq T$,
\begin{align}\label{eq:G}
       \|\bar{\G}(s,0;0,0)-1\|_k \leq C_{T,k}s^{1/4}.
\end{align}
Hence, we can combine with \eqref{eq:moments} to derive that 
\begin{align}\label{eq:I22}
\|I_2(t,s,x)\|_k
\les& p_t(x)^{-1} \int_\R p_{t-s}(x-y') p_s(y')\d y' s^{1/4} =s^{1/4}.
\end{align}
Now we give some estimate on the negative moment of $I_1$.  One realizes that $\int_\R \G(t,x;s,y)p_s(y)dy$ is the solution to SHE started from $p_s(y)$, evaluated at $(t-s,x)$, so by \cite[Corollary 4.9]{hule} we actually have 
\[
\|I_1(t,s,x)^{-1}\|^k_{k} \leq 2^k\e^{2k\sqrt{\lambda(t-s)\log \frac{2}{b(t-s)}}}
\left(1+4\sqrt{\pi k^2\lambda(t-s)}\e^{k^2\lambda(t-s)}\right),
\]
where the two functions $\lambda$ and $b$ are defined in (4.15) and (4.20) of \cite{hule} respectively. It is clear from \cite[Theorem 2.7(ii)]{hule} that both $\lambda$ and $1/b$ are locally bounded. Hence for $T>0$ and $k\geq2$, there exists $C_{T,k}'>0$ such that
\begin{align}\label{eq:I11}
\sup_{\substack{0<s<t\leq T\\ x\in\R}}\|I_1(t,s,x)^{-1}\|_k\leq C'_{T,k}.
\end{align}
By \eqref{eq:moments}, \eqref{eq:I22}, \eqref{eq:I11} and \eqref{eq:G}, we  obtain \[
\begin{aligned}
&\|err_1\|_k \les s^{1/4} \left(1
+s^{1/4}\right)\les s^{1/4},  \\
&\|err_3\|_k \les s^{1/4}.
\end{aligned}
\]

The analysis of $err_2$ is more complicated: we first rewrite it as 
\begin{equation}\label{e.err2}
\begin{aligned}
err_2=&\EE \left[\frac{\bar{\G}(t,x;s,y)-\bar{\G}(t,x;0,y)}{I_1(t,s,x) }\right]\\
&+\EE\left[ \frac{\bar{\G}(t,x;0,y)}{I_1(t,s,x)\bar{\G}(t,x;0,0)}(\bar{\G}(t,x;0,0)-I_1(t,s,x))\right].
\end{aligned}
\end{equation}
So we need to further estimate the error $\bar{\G}(t,x;s,y)-\bar{\G}(t,x;0,y)$ and $I_1(t,s,x)-\bar{\G}(t,x;0,0)$.    
First, by \cite[Lemma 3.6]{chris} we have
\begin{align}\label{eq:GG}
\|\bar{\G}(t,x;s,y)-\bar{\G}(t,x;0,y)\|_k\les |x-y|s^{1/4} 1_{|x-y|s^{1/4} <1}+1_{|x-y|s^{1/4} \geq 1}.
\end{align}
Secondly,  \[
\begin{aligned}
&I_1(t,s,x)-\bar{\G}(t,x;0,0)\\
&=p_t(x)^{-1}\int_\R \G(t,x;s,y')p_s(y')\d y'-p_t(x)^{-1} \G(t,x;0,0)\\
&=p_t(x)^{-1}\int_\R (\G(t,x;s,\sqrt{s/t}y')-\G(t,x;0,0))p_t(y')\d y'.
\end{aligned}
\]
To use the result from \cite{chris}, we rewrite the above term in terms of $\bar{\G}$:
\[
\begin{aligned}
&I_1(t,s,x)-\bar{\G}(t,x;0,0)\\
&=\int_\R (\bar{\G}(t,x;s,\sqrt{s/t}y')\tfrac{p_{t-s}(x-\sqrt{s/t}y')}{p_t(x)}-\bar{\G}(t,x;0,0)) p_t(y')\d y'\\
&=\int_\R \bar{\G}(t,x;s,\sqrt{s/t}y')(\tfrac{p_{t-s}(x-\sqrt{s/t}y')}{p_t(x)}-1) p_t(y')\d y'\\
&\quad +\int_\R (\bar{\G}(t,x;s,\sqrt{s/t}y')-\bar{\G}(t,x;0,0))p_t(y')\d y'=:I_3(t,s,x)+I_4(t,s,x).
\end{aligned}
\]

We  bound  $I_3$ and  $I_4$ as
\[
\begin{aligned}
&\|I_3(t,s,x)\|_k\les \int_\R |\tfrac{p_{t-s}(x-\sqrt{s/t}y)}{p_t(x)}-1| p_t(y)\d y\\
&\|I_4(t,s,x)\|_k \les \int_\R \|\bar{\G}(t,x;s,\sqrt{s/t}y)-\bar{\G}(t,x;0,0)\|_k p_t(y)\d y.
\end{aligned}
\]
\emph{For $I_3$,} after a calculation we have 
\[
|\tfrac{p_{t-s}(x-\sqrt{s/t}y)}{p_t(x)}-1| p_t(y)=\tfrac{1}{\sqrt{2\pi t}} |\tfrac{1}{\sqrt{(t-s)/t}}\e^{-\frac{(\sqrt{t}y-\sqrt{s}x)^2}{2(t-s)t}} -\e^{-\frac{y^2}{2t}}|,
\]
so it remains to estimate 
\[
\begin{aligned}
&\int_\R |\tfrac{1}{\sqrt{(t-s)/t}}\e^{-\frac{(\sqrt{t}y-\sqrt{s}x)^2}{2(t-s)t}} -\e^{-\frac{y^2}{2t}}|\d y \\
&\leq \int_\R |\tfrac{1}{\sqrt{(t-s)/t}}\e^{-\frac{(\sqrt{t}y-\sqrt{s}x)^2}{2(t-s)t}} -\tfrac{1}{\sqrt{(t-s)/t}}\e^{-\frac{y^2}{2(t-s)}}|\d y + \int_\R | \e^{-\frac{y^2}{2t}} -\tfrac{1}{\sqrt{(t-s)/t}}\e^{-\frac{y^2}{2(t-s)}}|\d y.
\end{aligned}
\]
For the second term, since $s$ is close to zero, we have by triangle inequality
\begin{align*}
&\int_\R | \e^{-\frac{y^2}{2  {t}}} -\tfrac{1}{\sqrt{(t-s)/t}}\e^{-\frac{y^2}{2(t-s)}}|\d y \\
&\qquad \leq 
\left|\frac{1}{\sqrt{(t-s)/t}}-1\right|\int_\R \e^{-\frac{y^2}{2t}}\d y
+ \frac{1}{\sqrt{(t-s)/t}}\int_\R \e^{-\frac{y^2}{2t}}
\left(1-\e^{-\frac{y^2s}{2(t-s)t}}\right)\d y
\les s.
\end{align*}
For the first term, we have 
$$\int_\R |\e^{-\frac{(y-\sqrt{s/t}x)^2}{2(t-s)}} -\e^{-\frac{y^2}{2(t-s)}}|\d y \les \sqrt{s}|x|1_{\sqrt{s}|x|\leq 1} +1_{\sqrt{s}|x|>1}.$$ 
This implies that 
\[
\|I_3(t,s,x)\|_k\les s+\sqrt{s}|x|1_{\sqrt{s}|x|\leq 1} +1_{\sqrt{s}|x|>1}.
\]
\emph{For $I_4$}, we have (note that $s\in (0,t/4)$)

\begin{align*}
&\|\bar{\G}(t,x;s,\sqrt{s/t}y)-\bar{\G}(t,x;0,0)\|_k \\
&\leq \|\bar{\G}(t,x;s,\sqrt{s/t}y)-\bar{\G}(t,x;s,0)\|_k +\|\bar{\G}(  {t},x;s,0)-\bar{\G}(  {t},x;0,0)\|_k\\
&  {=\|\bar{\G}(t-s,x-\sqrt{s/t}y;0,0)-\bar{\G}(t-s,x;0,0)\|_k +\|\bar{\G}(t-s,x;0,0)-\bar{\G}(t,x;0,0)\|_k},
\end{align*}
  {where the equality holds by the invariance property in \cite[Lemma 3.1]{chris}; see also the last line on page 24 of \cite{chris}. 
Then we apply \cite[Lemmas 3.4 and 3.6]{chris} to obtain}
\begin{align*}
\|\bar{\G}(t,x;s,\sqrt{s/t}y)-\bar{\G}(t,x;0,0)\|_k & \les |\sqrt{s}y|^{1/2}+ (|x|s^{1/4}1_{|x|s^{1/4}<1}+1_{|x|s^{1/4}\geq1}),
\end{align*}
which implies that 

\begin{align*}
\|I_4(t,s,x)\|_k &\les \int_\R ( |\sqrt{s}y|^{1/2}+ (|x|s^{1/4}1_{|x|s^{1/4}<1}+1_{|x|s^{1/4}\geq1})) p_t(y)\d y\\
&\les s^{1/4}+|x|s^{1/4}1_{|x|s^{1/4}<1}+1_{|x|s^{1/4}\geq1}.
\end{align*}

To summarize, we have 
\begin{align}\label{eq:I1G}
\|I_1(t,s,x)-\bar{\G}(  {t},x;0,0)\|_k &\les  s+\sqrt{s}|x|1_{\sqrt{s}|x|\leq 1} +1_{\sqrt{s}|x|>1} \nonumber\\
&+ s^{1/4}+|x|s^{1/4}1_{|x|s^{1/4}<1}+1_{|x|s^{1/4}\geq1}.
\end{align}
One can simplify the right-hand side so that  
\[
\|I_1(t,s,x)-\bar{\G}(  {t},x;0,0)\|_k \les s^{1/4}+|x|s^{1/4}1_{|x|s^{1/4}<1}+1_{|x|s^{1/4}\geq1}.
\]
Combining  \eqref{eq:GG}, \eqref{eq:I11}, \eqref{eq:I1G}, \eqref{eq:moments} with \eqref{e.err2}, we obtain 
\[
\begin{aligned}
\|err_2\|_k \les  (& |x-y|s^{1/4} 1_{|x-y|s^{1/4} <1}+1_{|x-y|s^{1/4} \geq 1}\\
&+   s^{1/4}+|x|s^{1/4}1_{|x|s^{1/4}<1}+1_{|x|s^{1/4}\geq1}).
\end{aligned}
\]
Therefore, it holds that
\[
\begin{aligned}
&\|\mathcal{A}(t,x;s,y)-g_t(x,y)\|_k\\
&\les( |x-y|s^{1/4} 1_{|x-y|s^{1/4} <1}+1_{|x-y|s^{1/4} \geq 1}+   s^{1/4}+|x|s^{1/4}1_{|x|s^{1/4}<1}+1_{|x|s^{1/4}\geq1}),
\end{aligned}
\]
which completes the proof.
\end{proof}

%
%

  \subsection{Spatial decay} 
  
 Recall that our goal is to determine the decaying rate of 
 \begin{equation}\label{e.6262}
    \begin{aligned}
    &\Cov[h(t,x),h(t,0)]\\
    &\qquad\qquad= \int_0^{t}\int_\R \EE [\mathcal{X}(t,s,x,y)]p_{s(t-s)/t}(y)p_{s(t-s)/t}(y-\frac{s}{t}x) \d y\d s. 
    \end{aligned}     \end{equation}  With Proposition~\ref{p.key} established in the previous section, one might expect that it is straightforward to approximate
  \[
  \mathcal{X}(t,s,x,y)= \mathcal{A}(t,x;s,y)  \mathcal{A}(t,0;s,y)\approx  g_t(x,y)g_t(0,y), \quad\quad \mbox{ for } s\ll1,
  \]  
    and then analyze the resulting integral. However, the situation turns out to be more delicate than it first appears.

    Using the following identities on the heat kernel
    \begin{align}
     p_t(x)p_t(y)=2p_{2t}(x+y)p_{2t}(x-y), \quad p_t(\alpha x)= \alpha^{-1}p_{t/\alpha^2}(x), 
    \end{align}
   for $x,y\in \R$ and $t, \alpha>0$,
we rewrite the product of the heat kernel in \eqref{e.6262} as 
\[
p_{s(t-s)/t}(y)p_{s(t-s)/t}(y-\tfrac{s}{t}x) =\tfrac12p_{s(t-s)/(2t)}(y-\tfrac{sx}{2t})p_{s(t-s)/(2t)}(\tfrac{sx}{2t}) ,
\]
so that the covariance can be rewritten as 
\begin{align*}
&\Cov[h(t,x),h(t,0)]\\
&\qquad\qquad=\tfrac12\int_0^t\int_\R \EE [\mathcal{X}(t,s,x,y)]p_{s(t-s)/(2t)}(y-\tfrac{sx}{2t})p_{s(t-s)/(2t)}(\tfrac{sx}{2t})\d y\d s.
\end{align*}

Assuming without loss of generality $x>0$, we do a change of variable 
\[
s\mapsto s/x^2, \quad\quad y\mapsto y/x,
\] so that the covariance can be rewritten as
\begin{align}\label{eq:covh}
&\Cov[h(t,x),h(t,0)] \nonumber\\
&\quad =\frac{t}{2x}\int_0^{tx^2}\int_\R\EE\left[ \mathcal{X}(t,\tfrac{s}{x^2},x,\tfrac{y}{x})\right]\frac{1}{\pi s(t-sx^{-2})}\e^{-\frac{t(y-s/(2t))^2}{s(t-sx^{-2})}}\e^{-\frac{s}{4t(t-sx^{-2})}} \d y\d s . 
\end{align}
The rest of this section is devoted to studying the integral on the right-hand side of \eqref{eq:covh} as we send $x\to\infty$. Throughout the section, $t>0$ is fixed (so the implicit multiplicative constant in all estimates may depend on $t$). 

\bigskip
We begin with the following lemma, which shows that it is sufficient to consider only small values of $s$.
\begin{lemma}\label{lem:appro1}
      {Fix $\beta \in (0, 2)$}. Then the integral
    \[
    \int_{x^\beta}^{tx^2}\int_\R\EE\left[ \mathcal{X}(t,\tfrac{s}{x^2},x,\tfrac{y}{x})\right]\frac{1}{s(t-sx^{-2})}
    \e^{-\frac{t(y-s/(2t))^2}{s(t-sx^{-2})}}\e^{-\frac{s}{4t(t-sx^{-2})}} \d y\d s
    \]
    tends to $0$ as $x\to\infty$.
\end{lemma}

\begin{proof}
  Denote the integral by $I$.  Since the expectation inside the above integral is uniformly bounded   {by \eqref{eq:moments}}, we bound it as    \[
    \begin{aligned}
     I&\les \int_{x^\beta}^{tx^2}\int_\R  \frac{1}{s(t-sx^{-2})}
    \e^{-\frac{t(y-s/(2t))^2}{s(t-sx^{-2})}}\e^{-\frac{s}{4t(t-sx^{-2})}} \d y\d s \\
   &   {=\sqrt{\frac{\pi}{t}} \int_{x^\beta}^{tx^2}\frac{1}{\sqrt{s(t-sx^{-2})}}\e^{-\frac{s}{4t(t-sx^{-2})}} \d s.}
    \end{aligned}
    \]
    Decompose the domain of integration into two parts, we have
    \[
 I\les     \int_{x^\beta}^{tx^2/2}\frac{1}{\sqrt{ s }}\e^{-\frac{s}{4t(t-x^{\beta-2})}}\d s+\int_{tx^2/2}^{tx^2} \frac{1}{s} \sqrt{\frac{s}{t(t-sx^{-2})}}\e^{-\frac{s}{4t(t-sx^{-2})}} \d s.
    \]
      {Clearly, the first integral converges to $0$ as $x\to\infty$. For the second integral, by the inequality $\sqrt{u}\e^{-u/4}\le 2\e^{-u/8}$ with $u=s/(t(t-sx^{-2}))$, we see that
    \begin{align*}
    \int_{tx^2/2}^{tx^2} \frac{1}{s} \sqrt{\frac{s}{t(t-sx^{-2})}}\e^{-\frac{s}{4t(t-sx^{-2})}} \d s&\leq 
    \int_{tx^2/2}^{tx^2} \frac{2}{s} \e^{-\frac{s}{8t(t-sx^{-2})}}\d s\\
    &=\int_{t/2}^{t} \frac{2}{s} \e^{-\frac{sx^2}{8t(t-s)}}\d s\to 0, \quad \text{as $x\to\infty$}.
    \end{align*}
    The proof is complete.}
\end{proof}

Next, we need to figure out the limit of 
\[
\EE\left[ \mathcal{X}(t,\tfrac{s}{x^2},x,\tfrac{y}{x})\right]
\]
as $x\to\infty$. Recall \[
 \mathcal{X}(t,\tfrac{s}{x^2},x,\tfrac{y}{x})=\mathcal{A}(t,x;\tfrac{s}{x^2},\tfrac{y}{x}) \mathcal{A}(t,0;\tfrac{s}{x^2},\tfrac{y}{x}),
\]
There are two terms on the right-hand side. We deal with the second one first: Proposition~\ref{p.key} immediately implies that (as $s/x^2,y/x\to0$) 
\[
 \mathcal{A}(t,0;\tfrac{s}{x^2},\tfrac{y}{x}) \to g_t(0,0)=1
\]
in $L^k(\Omega)$. This leads to the following lemma: 
\begin{lemma}\label{lem:appro2}
    Fix $\beta\in (0, 1/2) $. We have 
    \[
    \begin{aligned}
    \int_0^{x^\beta}\int_\R& \left(\EE\left[ \mathcal{X}(t,\tfrac{s}{x^2},x,\tfrac{y}{x})\right]-g_t(0,0)\EE\left[ \mathcal{A}(t,x;\tfrac{s}{x^2},\tfrac{y}{x})\right]\right)\\
    &\qquad\qquad\qquad\qquad \times  \frac{1}{s(t-sx^{-2})}
    \e^{-\frac{t(y-s/(2t))^2}{s(t-sx^{-2})}}\e^{-\frac{s}{4t(t-sx^{-2})}} \d y\d s \to 0
    \end{aligned}
    \]
    as $x\to\infty$.
\end{lemma}

\begin{proof}
 Using Cauchy-Schwarz inequality and \eqref{eq:moments}, we have 
\begin{align*}
&\left|\EE\left[ \mathcal{X}(t,\tfrac{s}{x^2},x,\tfrac{y}{x})\right]-g_t(0,0)\EE\left[ \mathcal{A}(t,x;\tfrac{s}{x^2},\tfrac{y}{x})\right]\right|\\ 
&\qquad\qquad\les \| \mathcal{A}(t,0;\tfrac{s}{x^2},\tfrac{y}{x})-g_t(0,0)\|_2\\
&\qquad\qquad\les \| \mathcal{A}(t,0;\tfrac{s}{x^2},\tfrac{y}{x})-g_t(0,\tfrac{y}{x})\|_2+  |g_t(0,\tfrac{y}{x})-g_t(0,0)|.
\end{align*}
 By Proposition~\ref{p.key}, the right-hand side can be bounded as 
    \begin{equation}\label{e.bd33}
    (sx^{-2})^{1/4}+|y/x|(sx^{-2})^{1/4}+|y/x|^{1/2},
    \end{equation}
    where the last term $|y/x|^{1/2}$ comes from the fact that  
      {\begin{align}\label{eq:holder2}
    |g_t(0,\tfrac{y}{x})-g_t(0,0)|\leq \|\bar{\G}(t,0;0,0)^{-1}\|_{4/3} \|\bar{\G}(t,0;0,0)-\bar{\G}(t,0;0,\tfrac{y}{x})\|_4\les |y/x|^{1/2},
    \end{align}
    (see \eqref{eq:moments} and \cite[Lemma 3.4]{chris}).}
    
    Recall that $\beta\in(0,\tfrac12)$. We substitute \eqref{e.bd33} into the integral and handle the three resulting terms separately.
      
  (i) 
    \begin{align}\label{eq:appro1}
    &\int_0^{x^\beta}\int_\R  (sx^{-2})^{1/4} \frac{1}{s(t-sx^{-2})}
    \e^{-\frac{t(y-s/(2t))^2}{s(t-sx^{-2})}}\e^{-\frac{s}{4t(t-sx^{-2})}} \d y\d s\nonumber\\
    & \qquad\qquad \les x^{-1/2}\int_0^{x^\beta}s^{-1/4}\d s \asymp x^{\frac{3\beta}{4}-\frac12}\to 0 \quad \text{as $x\to\infty$}
    \end{align}

    (ii)    \begin{align*}
    &\int_0^{x^\beta}\int_\R  (sx^{-2})^{1/4}|y/x| \frac{1}{s(t-sx^{-2})}
    \e^{-\frac{t(y-s/(2t))^2}{s(t-sx^{-2})}}\e^{-\frac{s}{4t(t-sx^{-2})}} \d y\d s\\
    & \qquad\qquad \les x^{-1/2}\int_0^{x^\beta}s^{-1/4} \frac{\sqrt{s}+s}{|x|}\d s \leq x^{\frac{7\beta}{4}-\frac32}\to 0 \quad \text{as $x\to\infty$}.
    \end{align*}

  (iii)
    \begin{align}\label{eq:appro3}
    &\int_0^{x^\beta}\int_\R  |y/x|^{1/2} \frac{1}{s(t-sx^{-2})}
    \e^{-\frac{t(y-s/(2t))^2}{s(t-sx^{-2})}}\e^{-\frac{s}{4t(t-sx^{-2})}} \d y\d s\nonumber\\
    & \qquad\qquad \les x^{-1/2}\int_0^{x^\beta}\frac{s^{1/4}+s^{1/2}}{\sqrt{s}}\d s \leq x^{\beta-\frac12}\to 0 \quad \text{as $x\to\infty$}.
    \end{align}
     Hence, we complete the proof.
\end{proof}

To analyze $\mathcal{A}(t,x;\tfrac{s}{x^2},\tfrac{y}{x})$ as $x\to\infty$,  we need the following lemma. 
\begin{lemma}\label{lem:asymp}
          For $t>s>0$,
          \begin{align*}
          \lim_{x\to\infty}\EE \left[\frac{\bar{\G}(t,x;\tfrac{s}{x^2},0)}{\bar{\G}(t,x;0,0)}\right]=1.
          \end{align*}
\end{lemma}
\begin{proof}
          According to Lemma \ref{lem:shift},
          \begin{align}\label{eq:denominator}
          \EE \left[\frac{\bar{\G}(t,x;\tfrac{s}{x^2},0)}{\bar{\G}(t,x;0,0)}\right] = \EE\left[\frac{\bar{\G}(t,0;\frac{s}{x^2},0)}{\int_\R p_{sx^{-2}(t-sx^{-2})/t}(z-\frac{s}{tx})\bar{\G}(t,0;\frac{s}{x^2},z)\bar{\G}(\frac{s}{x^2},z;0, 0)\d z}\right].
          \end{align}
          For the numerator, as $x\to\infty$,
          \begin{align*}
          \bar{\G}(t,0;\frac{s}{x^2},0) \to \bar{\G}(t,0;0,0)\quad\text{in $L^k(\Omega)$}.
          \end{align*}
          For the denominator, we first derive the following bound:  by Jensen's inequality, 
          \[
\begin{aligned}
&\bigg(\int_\R  p_{sx^{-2}(t-sx^{-2})/t}(z-\frac{s}{tx})\bar{\G}(t,0;\frac{s}{x^2},z)\bar{\G}(\frac{s}{x^2},z;0,0)\d z \bigg)^{-1}\\
&\qquad\leq  \int_\R  p_{sx^{-2}(t-sx^{-2})/t}(z-\frac{s}{tx})\big(\bar{\G}(t,0;\frac{s}{x^2},z)\bar{\G}(\frac{s}{x^2},z;0,0)\big)^{-1}\d z 
\end{aligned}
\]
which together with \eqref{eq:moments} implies that 
\[
\|\big(\int_\R  p_{sx^{-2}(  {t}-sx^{-2})/t}(z-\frac{s}{tx})\bar{\G}(t,0;\frac{s}{x^2},z)\bar{\G}(\frac{s}{x^2},  {z};0,0)\d z\big)^{-1}\|_k \les 1.
\]

To complete the proof, it suffices to show that the denominator on the right-hand side of \eqref{eq:denominator} converges to $\bar{\G}(t,0;0,0)$ in $L^k(\Omega)$ as $x\to\infty$. 
    We write the error as 
    \[
    \begin{aligned}
    &\int_\R  p_{sx^{-2}(t-sx^{-2})/t}(z-\frac{s}{tx})\bar{\G}(t,0;\frac{s}{x^2},z)\bar{\G}(\frac{s}{x^2},z;0,0)\d z- \bar{\G}(t,0;0,0)\\
    &=\int_\R  p_{sx^{-2}(t-sx^{-2})/t}(z-\frac{s}{tx})[\bar{\G}(t,0;\frac{s}{x^2},z)-\bar{\G}(t,0;0,0)]\bar{\G}(\frac{s}{x^2},z;0,0)\d z\\
    &\quad+\int_\R  p_{sx^{-2}(t-sx^{-2})/t}(z-\frac{s}{tx})\bar{\G}(t,0;0,0)[\bar{\G}(\frac{s}{x^2},z;0,0)-1]\d z=:J_1+J_2.
    \end{aligned}
    \]
    For $J_1$, 
      {we first write
    \begin{align*}
    \|\bar{\G}(t,0;\frac{s}{x^2},z)-\bar{\G}(t,0;0,0)\|_{2k} &\leq 
    \|\bar{\G}(t,0;\frac{s}{x^2},z)-\bar{\G}(t,0;0,z)\|_{2k}
    +\|\bar{\G}(t,0;0,z)-\bar{\G}(t,0;0,0)\|_{2k}\\
    &=\|\bar{\G}(t-\frac{s}{x^2},-z;0,0)-\bar{\G}(t,-z;0,0)\|_{2k}
    +\|\bar{\G}(t,-z;0,0)-\bar{\G}(t,0;0,0)\|_{2k},
    \end{align*}
    where the equality holds by the shift invariance of the Green's function for SHE (see \cite[Lemma 3.1]{chris}) and we also  used a time reversal.
    Next, Lemma 3.4 of \cite{chris} ensures that
    \begin{align*}
     \|\bar{\G}(t,-z;0,0)-\bar{\G}(t,0;0,0)\|_{2k} \les |z|^{1/2}.
    \end{align*}
    In addition, we conclude from \cite[Lemma 3.6]{chris} that
    \begin{align*}
     \|\bar{\G}(t-\frac{s}{x^2},-z;0,0)-\bar{\G}(t,-z;0,0)\|_{2k}
     \les |z|(sx^{-2})^{1/4}.
    \end{align*}
    Thus, it follows that
    }
    \[
    \|J_1\|_k \les \int_\R p_{sx^{-2}(t-sx^{-2})/t}(z-\frac{s}{tx})[|z|(sx^{-2})^{1/4}+|z|^{1/2}] \d z\to0
    \]
    as $x\to\infty$. For $J_2$, by \eqref{eq:G}, we have 
    \[
    \|J_2\|_k\les \int_\R p_{sx^{-2}(t-sx^{-2})/t}(z-\frac{s}{tx})(sx^{-2})^{1/4} \d z=(sx^{-2})^{1/4}\to0
    \]
    as $x\to\infty$.
    The proof is complete.        
\end{proof}

With the above preparations, we are now ready to prove Theorem \ref{th:cov}.

\begin{proof}[Proof of Theorem \ref{th:cov}]
Fix $\beta\in (0, 1/2)$.
By \eqref{eq:covh} and  Lemmas~\ref{lem:appro1} and \ref{lem:appro2},  as $x\to\infty$,
\[
\begin{aligned}
&\frac{2x}{t}\Cov[h(t,x),h(t,0)]\\
&= \int_0^{x^\beta} \int_\R \EE \left[\mathcal{A}(t,x;\tfrac{s}{x^2},\tfrac{y}{x})\right] 
 \frac{1}{\pi s(t-sx^{-2})}\e^{-\frac{t(y-s/(2t))^2}{s(t-sx^{-2})}} \e^{-\frac{s}{4t(t-sx^{-2})}} \d y\d s +o(1).
\end{aligned}
\]
Since
\[
\EE\left[\mathcal{A}(t,x;\tfrac{s}{x^2},\tfrac{y}{x})\right]=\EE \left[\frac{\bar{\G}(t,x;\tfrac{s}{x^2},\tfrac{y}{x})\bar{\G}(\tfrac{s}{x^2},\tfrac{y}{x};0,0)}{\bar{\G}(t,x;0,0)}\right]
\]
and
\[
\|\bar{\G}(\tfrac{s}{x^2},\tfrac{y}{x};0,0)- 1\|_k=\|\bar{\G}(\tfrac{s}{x^2},0;0,0)- 1\|_k \les |\tfrac{s}{x^2}|^{1/4},
\]
we perform the same calculation as in  \eqref{eq:appro1} to derive that
\[
\begin{aligned}
&\frac{2x}{t}\Cov[h(t,x),h(t,0)]\\
&= \int_0^{x^\beta} \int_\R \EE \left[\frac{\bar{\G}(t,x;\tfrac{s}{x^2},\tfrac{y}{x})}{\bar{\G}(t,x;0,0)}\right]
 \frac{1}{\pi s(t-sx^{-2})}\e^{-\frac{t(y-s/(2t))^2}{s(t-sx^{-2})}} \e^{-\frac{s}{4t(t-sx^{-2})}} \d y\d s +o(1).
\end{aligned}
\]
Furthermore, because
\begin{align*}
\|\bar{\G}(t,x;\tfrac{s}{x^2},\tfrac{y}{x}) -\bar{\G}(t,x;\tfrac{s}{x^2},0)\|_k\les |y/x|^{1/2},
\end{align*}
we deduce from the calculation in \eqref{eq:appro3} that
\[
\begin{aligned}
&\frac{2x}{t}\Cov[h(t,x),h(t,0)]\\
&= \int_0^{x^\beta} \int_\R \EE \left[\frac{\bar{\G}(t,x;\tfrac{s}{x^2},0)}{\bar{\G}(t,x;0,0)}\right]
 \frac{1}{\pi s(t-sx^{-2})}\e^{-\frac{t(y-s/(2t))^2}{s(t-sx^{-2})}} \e^{-\frac{s}{4t(t-sx^{-2})}} \d y\d s +o(1)\\
 &=\int_0^{x^\beta} \EE \left[\frac{\bar{\G}(t,x;\tfrac{s}{x^2},0)}{\bar{\G}(t,x;0,0)}\right]
 \frac{1}{\sqrt{\pi st(t-sx^{-2})}}\e^{-\frac{s}{4t(t-sx^{-2})}} \d s +o(1).
\end{aligned}
\]
  {Since the  scaled Green's function of SHE has finite uniform positive and negative moments (see \eqref{eq:moments}),  the expectation inside the above integral is uniformly bounded in $s$ and $x$. }Hence, by Lemma \ref{lem:asymp} and dominated convergence theorem, 
\begin{align*}
\lim_{x\to\infty}\left[\frac{2x}{t}\Cov[h(t,x),h(t,0)]\right] = \int_0^\infty \frac{1}{\sqrt{\pi st^2}}\e^{-\frac{s}{4t^2}}\d s = 2, 
\end{align*}
which completes the proof.
\end{proof}

\begin{corollary}\label{cor:cov}
          Fix $t>0$. Then
          \begin{align*}
          \lim_{N\to\infty}\frac{1}{N\log N}\EE\left[\left(\int_0^N(h(t,x)-\EE[h(t,x)])\d x\right)^2\right]=2t. 
          \end{align*}
\end{corollary}
\begin{proof}
          Since $\{Z(t,x)/p_t(x): x\in \R\}$ is stationary, we can write
          \begin{align*}
          &\frac{1}{N\log N}\EE\left[\left(\int_0^N(h(t,x)-\EE[h(t,x)])\d x\right)^2\right]\\
          &\qquad =\frac{1}{N\log N}\int_{[0, N]^2}\d x\d y\, \Cov[h(t,x), h(t,y)]\\
                    &\qquad =\frac{2}{N\log N}\int_0^N\d y\int_0^y\d x\, \Cov\left[\log \frac{Z(t,x)}{p_t(x)}, \log \frac{Z(t,y)}{p_t(y)}\right]\\
                                        &\qquad =\frac{2}{N\log N}\int_0^N\d y\int_0^y\d x\, \Cov[h(t,x), h(t,0)].
          \end{align*}
          Then the result follows from Theorem \ref{th:cov}.
\end{proof}

\section{CLT for spatial average: proof of Theorem \ref{th:CLT}}
\label{s.clt}
We will prove Theorem \ref{th:CLT} in this section. For $t>0$ and $N\geq1$, we denote
    \begin{align}\label{eq:X_Nt}
    X_N(t)=\frac{1}{\sqrt{N\log N}}\int_{0}^N [h(t,x)-\EE [h(t,x)]]\d x.
    \end{align}
    The goal is to show $X_N(\cdot)\Rightarrow \sqrt{2}{\rm B}(\cdot)$ in the sense of convergence of finite dimensional distributions. 
    
    Recall that by the Clark-Ocone formula, we have expressed $h(t,x)-\EE[h(t,x)]$ as an It\^o integral:
    \[
    h(t,x)-\EE[ h(t,x)]=\int_0^t\int_\R \A(t,x;s,y)p_{s(t-s)/t}(y-\frac{s}{t}x)\,\xi(\d s\, \d y),
    \]
    with $\A$ defined in \eqref{eq:A}. This implies through an application of Fubini theorem that 
    \[
    X_N(t)=\frac{1}{\sqrt{N\log N}}\int_0^t\int_\R \int_0^N\A(t,x;s,y)p_{s(t-s)/t}(y-\frac{s}{t}x) \d x\, \xi(\d s\, \d y).
    \]
    So the quantity of interest is written as an It\^o integral, which we can actually analyze directly. Here we take a different approach: define
      \begin{align}\label{eq:X_N1}
    X_{N,1}(t)= \frac{1}{\sqrt{N\log N}}\int_0^t\int_\R \int_0^N\EE[\A(t,x;s,y)]p_{s(t-s)/t}(y-\frac{s}{t}x) \d x\, \xi(\d s\, \d y).
    \end{align}
    By Stroock's formula (see \cite[Equation (7), page 3]{stroock}), 
     $X_{N,1}(t)$ is the first Wiener chaos component of $X_{N}(t)$. 
       {
    More precisely, Stroock's formula implies that the first Wiener chaos component of $X_{N}(t)$ is given by the  Wiener integral $\int_0^t\int_\R \EE[D_{s,y}X_{N}(t)]\xi(\d s\, \d y)$. Since the Malliavin derivative of $X_N(t)$ is given by
    \begin{align*}
    D_{s,y} X_N(t)&= \frac{1}{\sqrt{N\log N}}
    \int_0^N\A(t,x;s,y)p_{s(t-s)/t}(y-\frac{s}{t}x) \d x \\
    &\qquad +
    \frac{1}{\sqrt{N\log N}}\int_s^t\int_\R \int_0^N D_{s,y}\A(t,x;r,z)p_{r(t-r)/t}(y-\frac{r}{t}x) \d x\, \xi(\d r\, \d z),
    \end{align*}
    it follows that
    \begin{align*}
     \EE[D_{s,y} X_N(t)]&=\frac{1}{\sqrt{N\log N}}
    \int_0^N\EE[\A(t,x;s,y)]p_{s(t-s)/t}(y-\frac{s}{t}x) \d x,
    \end{align*}
    which implies that the first Wiener chaos component of $X_{N}(t)$
    is equal to $X_{N,1}(t)$ defined in \eqref{eq:X_N1}.
     }
     We claim that $X_{N,1}(t)$ is the main contribution to $X_N(t)$ as $N\to\infty$,  which in turn yields both the Gaussianity and the convergence of the variance.   {In addition, since $X_{N,1}(\cdot)$ is a Gaussian process , its f.d.d limit is determined by the limit of the  covariance function}. This is formalized in the following result:

\begin{proposition}\label{prop:twotime}
      Fix $t_1, t_2>0$. Then 
      \begin{align}\label{eq:s0}
      \lim_{N\to\infty}\Cov[X_{N,1}(t_1), X_{N,1}(t_2)]=2(t_1\wedge t_2).
      \end{align}
\end{proposition}

With the above result, 
we now proceed with a quick proof of Theorem~\ref{th:CLT}.

\begin{proof}[Proof of Theorem \ref{th:CLT}]
          With $X_{N}(t)$ defined in \eqref{eq:X_Nt}, we write its Wiener chaos expansion: 
          $$X_{N}(t) = \sum_{k=1}^{\infty}X_{N,k}(t),$$
          where $X_{N,k}(t)$ is the $k$-th Wiener chaos component of $X_N(t)$. In particular, $X_{N,1}(t)$ is given in \eqref{eq:X_N1}. 
          By Corollary \ref{cor:cov} and Proposition \ref{prop:twotime}, we conclude that  for each $t>0$
          \begin{align*}
          \lim_{N\to\infty} \EE[X_N(t)^2]=2t\quad\text{and}\quad \lim_{N\to\infty} \EE[X_{N,1}(t)^2]=2t.
          \end{align*}
          Thus, by the $L^2(\Omega)$-orthogonality of the components of the Wiener chaos,
          \begin{align*}
          \EE\left[(X_N(t)-X_{N,1}(t))^2\right] = \EE\left[X_{N}(t)^2\right] - \EE\left[X_{N,1}(t)^2\right]  \to 0
          \end{align*}
          as $N\to\infty$. Therefore, Theorem \ref{th:CLT} holds by Cram\'er–Wold theorem and Proposition \ref{prop:twotime}.
\end{proof}

\subsection{Proof of Proposition~\ref{prop:twotime}}
%
%


 Applying Ito isometry we write the covariance as
 \begin{align*}
 & \Cov[X_{N,1}(t_1), X_{N,1}(t_2)]  \\
&=\frac{1}{N\log N}\int_{[0, N]^2}\d x_1\d x_2 \int_{0}^{t_1\wedge t_2} \d s\int_\R \d y\,
       \EE[\A(t_1,x_1;s,y)] \EE[\A(t_2,x_2;s,y)]\\
       &\qquad\qquad\qquad \qquad\qquad p_{s(t_1-s)/t_1}(y-\frac{s}{t_1}x_1)
       p_{s(t_2-s)/t_2}(y-\frac{s}{t_2}x_2).
      \end{align*}
      
      The rest of proof is a rather complicated calculation.

Recall that 
\[
 \A(t,x;s,y)=\EE\left[\frac{\bar{\G}(t,x;s,y)\bar\G(s,y;0,0)}{\bar{\G}(t,x;0,0)}|\F_s\right].
\]
Since the main contribution to the stochastic integral \eqref{eq:X_N1} comes from those small $s\ll1$, we define
    \begin{align*}
     \B(t,x;s,y)=\frac{\bar{\G}(t,x;s,y)}{\bar{\G}(t,x;0,0)}
     \end{align*}
    and 
    \begin{align*}
    Y_{N,1}(t)= \frac{1}{\sqrt{N\log N}}\int_0^t\int_\R \int_0^N\EE[\B(t,x;s,y)]p_{s(t-s)/t}(y-\frac{s}{t}x) \d x\, \xi(\d s\, \d y).
    \end{align*}
       By the moment estimates on 
        $\A(t,x;s,y)$ and $\B(t,x;s,y)$, \eqref{eq:G} and Lemma \ref{lem:s0}, it suffices to prove
      \begin{align}\label{eq:covY}
      \lim_{N\to\infty}\Cov(Y_{N,1}(t_1), Y_{N,1}(t_2))=2(t_1\wedge t_2).
      \end{align}

      Define 
      \begin{align*}
       &V_{N}(t_1,t_2)= \int_{[0, N]^2}\d x_1\d x_2 \int_{(t_1\wedge t_2)/N^2}^{t_1\wedge t_2} \d s\int_\R \d y\,
       \EE[\B(t_1,x_1;s,y)] \EE[\B(t_2,x_2;s,y)]\\
       &\qquad\qquad\qquad \qquad\qquad \times p_{s(t_1-s)/t_1}(y-\frac{s}{t_1}x_1)
       p_{s(t_2-s)/t_2}(y-\frac{s}{t_2}x_2).
      \end{align*}
      Applying Lemma \ref{lem:2}, proving \eqref{eq:covY} reduces to  
      \begin{align}\label{eq:cov}
       &\lim_{N\to\infty}\frac{V_N(t_1,t_2)}{N\log N}=2(t_1\wedge t_2).
      \end{align}
      By a change of variable $x_i\mapsto x_it_i,i=1,2$ and $y\mapsto ys$, we can write
      \begin{align*}
       &V_N(t_1,t_2)= t_1t_2\int_0^{N/t_1}\d x_1\int_0^{N/t_2}\d x_2  \int_{(t_1\wedge t_2)/N^2}^{t_1\wedge t_2} \frac{\d s}{s}\int_\R \d y\,
        p_{\frac{1}{s}-\frac{1}{t_1}}(y-x_1)
       p_{\frac{1}{s}-\frac{1}{t_2}}(y-x_2)
       \\
       &\qquad\qquad\qquad \qquad\qquad \times\EE[\B(t_1,t_1x_1;s,sy)] \EE[\B(t_2,t_2x_2;s,sy)].
      \end{align*} 
      Further using the change of variable of $s\to N^{-\tau}(t_1\wedge t_2), y\to y+x_2$, it follows that
      \begin{align*}
      &\frac{V_N(t_1,t_2)}{N\log N}= \frac{t_1t_2}{N}
      \int_0^{N/t_1}\d x_1\int_0^{N/t_2}\d x_2  \int_0^2\d \tau\int_\R \d y\,
        p_{\frac{N^{\tau}}{t_1\wedge t_2}-\frac{1}{t_1}}(y+x_2-x_1)
       p_{\frac{N^\tau}{t_1\wedge t_2}-\frac{1}{t_2}}(y)       \\
       &\qquad\qquad\qquad \qquad\qquad \times \EE[\B(t_1,t_1x_1;N^{-\tau}(t_1\wedge t_2),N^{-\tau}(t_1\wedge t_2)(y+x_2))] \\
              &\qquad\qquad\qquad \qquad\qquad\times \EE[\B(t_2,t_2x_2;N^{-\tau}(t_1\wedge t_2),N^{-\tau}(t_1\wedge t_2)(y+x_2))].
      \end{align*}      
       
       In order to derive \eqref{eq:cov}, we will approximate the  two expectations appearing in the above integral by $1$ as $N\to\infty$. The remainder of the proof consists of several steps to justify this approximation.
       
       Step 1. We first deal with the second expectation
       \[
       \EE[\B(t_2,t_2x_2;N^{-\tau}(t_1\wedge t_2),N^{-\tau}(t_1\wedge t_2)(y+x_2))].
       \]
      To simplify the notation, we define
      \begin{align*}
       \lambda(N)= N^{-\tau}(t_1\wedge t_2)(t_2-N^{-\tau}(t_1\wedge t_2))/t_2
      \end{align*}
      and
      \begin{align*}
       \D(N, \tau, y,x_2)&=\int_\R p_{\lambda(N)}(z+ N^{-\tau}(t_1\wedge t_2)y)\bar{\G}(t_2,0; N^{-\tau}(t_1\wedge t_2), z)\\
       &\qquad\qquad\qquad \bar{\G}
       (N^{-\tau}(t_1\wedge t_2), z+ N^{-\tau}(t_1\wedge t_2)(y+x_2);0,0)\d z.
      \end{align*}
      According to Lemma \ref{lem:shift}, we have 
      \begin{align*}
       &\EE[\B(t_2,t_2x_2;N^{-\tau}(t_1\wedge t_2),N^{-\tau}(t_1\wedge t_2)(y+x_2))]
       = \EE\left[\frac{\bar{\G}(t_2, 0; N^{-\tau}(t_1\wedge t_2),0)}{\D(N, \tau, y,x_2)}\right].
      \end{align*}
        {The term $\D(N, \tau, y,x_2)$ may be viewed as a spatial average (by Gaussian kernel) of the factor $\bar{\G}(\cdot)\bar{\G}(\cdot)$ in the $z-$variable, so by Jensen's inequality and the fact that $\bar{\G}(\cdot)$ has negative moments bound of all order}, we derive that $\D(N, \tau, y,x_2)$ has finite negative moments, bounded uniformly in $(N, \tau, y,x_2)$. By   {\cite[Lemma 3.6]{chris}},    \begin{align}\label{eq:holdert}
      \|\bar{\G}(t_2, 0; N^{-\tau}(t_1\wedge t_2),0) - \bar{\G}(t_2, 0; 0,0)\|_k\lesssim N^{-\tau/4},
      \end{align}
      which implies 
         \begin{align*}
       &\EE[\B(t_2,t_2x_2;N^{-\tau}(t_1\wedge t_2),N^{-\tau}(t_1\wedge t_2)(y+x_2))]
       = \EE\left[\frac{\bar{\G}(t_2, 0; 0,0)}{\D(N, \tau, y,x_2)}\right] + O(N^{-\tau/4}).
      \end{align*} 
      Define 
      \begin{align*}
       \D_1(N, \tau, y,x_2)&=\int_\R p_{\lambda(N)}(z+ N^{-\tau}(t_1\wedge t_2)y)\bar{\G}(t_2,0; N^{-\tau}(t_1\wedge t_2), z)\d z.
      \end{align*}
       Similar to $\D(N, \tau, y,x_2)$, the random variable $\D_1(N, \tau, y,x_2)$ has finite negative moments, bounded uniformly in $(N, \tau, y,x_2)$.
      Appealing to \eqref{eq:G} and using stationarity, 
      \begin{align*}
      \|\D(N, \tau, y,x_2)-\D_1(N, \tau, y,x_2)\|_k \les N^{-\tau/4}.
      \end{align*}
      Notice that
      \begin{align*}
       \D_1(N, \tau, y,x_2)= \int_\R p_1(z)\bar{\G}(t_2,0; N^{-\tau}(t_1\wedge t_2), z\sqrt{\lambda(N)} - N^{-\tau}(t_1\wedge t_2)y)\d z.
      \end{align*}
      By triangle inequality and \cite[Lemmas 3.4 and 3.6]{chris},
      \begin{align*}
       &\|\bar{\G}(t_2,0; N^{-\tau}(t_1\wedge t_2), z\sqrt{\lambda(N)} - N^{-\tau}(t_1\wedge t_2)y)
       - \bar{\G}(t_2,0; 0,0)\|_k\\
       &\qquad \leq \|\bar{\G}(t_2,0; N^{-\tau}(t_1\wedge t_2), z\sqrt{\lambda(N)} - N^{-\tau}(t_1\wedge t_2)y)
       - \bar{\G}(t_2,0; N^{-\tau}(t_1\wedge t_2), 0)\|_k \\
       &\qquad \quad + \|\bar{\G}(t_2,0; N^{-\tau}(t_1\wedge t_2), 0)
       - \bar{\G}(t_2,0; 0,0)\|_k\\
       &\qquad \lesssim |z|^{1/2}N^{-\tau/4} + (1\wedge |N^{-\tau}y|^{1/2})+ N^{-\tau/4}.
      \end{align*}
      Thus, we obtain that
      \begin{align*}
      \|\D_1(N, \tau, y,x_2)-\bar{\G}(t_2,0; 0,0)\|_k\les (1\wedge |N^{-\tau}y|^{1/2}) +N^{-\tau/4},
      \end{align*}
      which implies that
      \begin{align*}
       &\EE[\B(t_2,t_2x_2;N^{-\tau}(t_1\wedge t_2),N^{-\tau}(t_1\wedge t_2)(y+x_2))]
       = 1 + O((1\wedge |N^{-\tau}y|^{1/2}) +N^{-\tau/4}).
      \end{align*} 
      Applying Lemma \ref{lem:y} below, we conclude   that as $N\to\infty$,
      \begin{align}
      \frac{V_N(t_1,t_2)}{N\log N}&=\frac{t_1t_2}{N}
      \int_0^{N/t_1}\d x_1\int_0^{N/t_2}\d x_2  \int_0^2\d \tau\int_\R \d y\,
        p_{\frac{N^{\tau}}{t_1\wedge t_2}-\frac{1}{t_1}}(y+x_2-x_1)
       p_{\frac{N^\tau}{t_1\wedge t_2}-\frac{1}{t_2}}(y)   \nonumber    \\
       &\qquad\qquad\qquad \qquad\qquad \EE[\B(t_1,t_1x_1;N^{-\tau}(t_1\wedge t_2),N^{-\tau}(t_1\wedge t_2)(y+x_2))]\label{eq:expectation}\\
       &\quad +o(1).\nonumber
      \end{align}

      Step 2. We continue to approximate the expectation appearing in \eqref{eq:expectation} by $1$ as $N\to\infty$.

      First, because
      \begin{align*}
       &\big|\EE[\B(t_1,t_1x_1;N^{-\tau}(t_1\wedge t_2),N^{-\tau}(t_1\wedge t_2)(y+x_2))]\\
       &\qquad\qquad\qquad\qquad-\EE[\B(t_1,t_1x_1;N^{-\tau}(t_1\wedge t_2),N^{-\tau}(t_1\wedge t_2)x_2)]\big|\lesssim 1\wedge|N^{-\tau}y|^{1/2},
      \end{align*}
      we obtain by Lemma \ref{lem:y} that as $N\to\infty$
            \begin{align}\label{eq:v1}
      \frac{V_N(t_1,t_2)}{N\log N}&=o(1)+\frac{t_1t_2}{N}
      \int_0^{N/t_1}\d x_1\int_0^{N/t_2}\d x_2  \int_0^2\d \tau\,
        p_{\frac{2N^{\tau}}{t_1\wedge t_2}-\frac{1}{t_1}-\frac{1}{t_2}}(x_2-x_1)\nonumber\\
       &\qquad\qquad\qquad \qquad\qquad \EE[\B(t_1,t_1x_1;N^{-\tau}(t_1\wedge t_2),N^{-\tau}(t_1\wedge t_2)x_2)].
       \end{align}
           Next, we introduce
      \begin{align*}
       \tilde{\lambda}(N)= N^{-\tau}(t_1\wedge t_2)(t_1-N^{-\tau}(t_1\wedge t_2))/t_1,
      \end{align*}
      and
      \begin{align*}
       \tilde{\D}(N, \tau,x_1,x_2)&=\int_\R p_{\tilde{\lambda}(N)}(z+ N^{-\tau}(t_1\wedge t_2)(x_2-x_1))\bar{\G}(t_1,0; N^{-\tau}(t_1\wedge t_2), z)\\
       &\qquad\qquad\qquad \bar{\G}
       (N^{-\tau}(t_1\wedge t_2), z+ N^{-\tau}(t_1\wedge t_2)x_2;0,0)\d z.
      \end{align*}
      Then, by Lemma \ref{lem:shift},
      \begin{align*}
       \EE[\B(t_1,t_1x_1;N^{-\tau}(t_1\wedge t_2),N^{-\tau}(t_1\wedge t_2)x_2)]
       &=\EE\left[\frac{\bar{\G}(t_1, 0; N^{-\tau}(t_1\wedge t_2),0)}{\tilde{\D}(N, \tau,x_1,x_2)}\right]\\
       &=\EE\left[\frac{\bar{\G}(t_1, 0; 0,0)}{\tilde{\D}(N, \tau,x_1,x_2)}\right]+ O(N^{-\tau/4}),
      \end{align*}
      where the second equality holds by \eqref{eq:holdert}.

      We will approximate the denominator in the above expectation by $\bar\G(t_1,0;0,0)$ using the same arguments as before. By \eqref{eq:holdert} and stationarity
      \begin{align*}
       &\left\|\tilde{\D}(N, \tau,x_1,x_2)-\int_\R p_{\tilde{\lambda}(N)}(z+ N^{-\tau}(t_1\wedge t_2)(x_2-x_1))
       \bar{\G}(t_1,0; N^{-\tau}(t_1\wedge t_2),z)\d z\right\|_k
       \les N^{-\tau/4}
       \end{align*}
       which implies that
       \begin{align*}
       &\left\|\tilde{\D}(N, \tau,x_1,x_2)-\int_\R p_1(z)\bar{\G}(t_1,0; N^{-\tau}(t_1\wedge t_2),z\sqrt{\tilde{\lambda}(N)}-N^{-\tau}(t_1\wedge t_2)(x_2-x_1))\d z\right\|_k\\
       & \quad \les N^{-\tau/4}.
       \end{align*}
       Because
       \begin{align*}
       &\left\|\bar{\G}(t_1,0; N^{-\tau}(t_1\wedge t_2),z\sqrt{\tilde{\lambda}(N)}-N^{-\tau}(t_1\wedge t_2)(x_2-x_1))-\bar{\G}(t_1,0;0,0)\right\|_k\\
       &\qquad \qquad \qquad\qquad\qquad\qquad\les N^{-\tau/4}+
       \left(1\wedge \left|z\sqrt{\tilde{\lambda}(N)}-N^{-\tau}(t_1\wedge t_2)(x_2-x_1)\right|^{1/2}\right),
       \end{align*}
       it follows that
       \begin{align*}
       &\left\|\tilde{\D}(N, \tau,x_1,x_2)- \bar{\G}(t_1,0;0,0)\right\|_k\\
      &\qquad\qquad\qquad\les
                    N^{-\tau/4}
                    + \int_\R p_1(z)\left[1\wedge \left|z\sqrt{\tilde{\lambda}(N)}-N^{-\tau}(t_1\wedge t_2)(x_2-x_1)\right|^{1/2}\right]\d z.
       \end{align*}
       
       Hence, since $\tilde{\D}(N, \tau,x_1,x_2)$ has finite negative moments, bounded uniformly in $(N, \tau,x_1,x_2)$, we deduce that
       \begin{align}\label{eq:v2}
       &\left|\EE[\B(t_1,t_1x_1;N^{-\tau}(t_1\wedge t_2),N^{-\tau}(t_1\wedge t_2)x_2)]-1\right|\nonumber\\
       &\qquad \qquad \qquad \lesssim 
       \int_\R p_1(z)\left[1\wedge \left|z\sqrt{\tilde{\lambda}(N)}-N^{-\tau}(t_1\wedge t_2)(x_2-x_1)\right|^{1/2}\right]\d z
                    +N^{-\tau/4}.
       \end{align}
       
       Step 3. Replacing the expectation in \eqref{eq:v1} by the integral in \eqref{eq:v2}, we claim that the resulting quantity converges to $0$ as $N\to\infty$.

       To see this, introduce the following two functions
       \begin{align*}
       1_N(x):=\frac1N1_{[0,N]}(x) \quad \text{and} \quad \tilde{1}_N(x):=1_N(-x), \quad x\in \R.
\end{align*}
       Denote $t^*=\frac{1}{t_1\wedge t_2}$ and write
       \begin{align*}
        &\frac{1}{Nt^*}
      \int_{[0, Nt^*]^2}\d x_1\d x_2 \int_0^2\d \tau\,
        p_{\frac{2N^{\tau}}{t_1\wedge t_2}-\frac{1}{t_1}-\frac{1}{t_2}}(x_2-x_1)\\
       &\qquad\qquad\qquad\qquad\qquad \int_\R p_1(z)\left[1\wedge \left|z\sqrt{\tilde{\lambda}(N)}-N^{-\tau}(t_1\wedge t_2)(x_2-x_1)\right|^{1/2}\right]\d z\\
       &\quad = Nt^*
      \int_\R \d y \left(1_{Nt^*}*\tilde{1}_{Nt^*}\right)(y) \int_0^2\d \tau\,
        p_{\frac{2N^{\tau}}{t_1\wedge t_2}-\frac{1}{t_1}-\frac{1}{t_2}}(y)\\
       &\qquad\qquad\qquad\qquad\qquad \int_\R p_1(z)\left[1\wedge \left|z\sqrt{\tilde{\lambda}(N)}-N^{-\tau}(t_1\wedge t_2)y\right|^{1/2}\right]\d z\\
       &\quad =  Nt^*
      \int_\R \d y \left(1_{Nt^*}*\tilde{1}_{Nt^*}\right)(y\sqrt{\frac{2N^{\tau}}{t_1\wedge t_2}-\frac{1}{t_1}-\frac{1}{t_2}}) \int_0^2\d \tau\,
        p_{1}(y)\\
       &\qquad\qquad\qquad \int_\R p_1(z)\left[1\wedge \left|z\sqrt{\tilde{\lambda}(N)}-N^{-\tau}(t_1\wedge t_2)y
       \sqrt{\frac{2N^{\tau}}{t_1\wedge t_2}-\frac{1}{t_1}-\frac{1}{t_2}}
       \right|^{1/2}\right]\d z.
       \end{align*}
       Because the function $Nt^*(1_{Nt^*}*\tilde{1}_{Nt^*})$ is bounded above by $1$, the above integral converges to $0$ as $N\to\infty$ by the dominated
       convergence theorem.

       Finally, we conclude from \eqref{eq:v1} and \eqref{eq:v2} that
       \begin{align*}
       \lim_{N\to\infty}\frac{V_N(t_1,t_2)}{N\log N} &= \lim_{N\to\infty}\left[ \frac{t_1t_2}{N}
      \int_0^{N/t_1}\d x_1\int_0^{N/t_2}\d x_2  \int_0^2\d \tau\,
        p_{\frac{2N^{\tau}}{t_1\wedge t_2}-\frac{1}{t_1}-\frac{1}{t_2}}(x_2-x_1)\right]\\
        &= 2(t_1\wedge t_2).
       \end{align*}
       where the second equality holds by Lemma \ref{lem:twotime}.  This proves \eqref{eq:cov} and hence completes the proof of Proposition \ref{prop:twotime}.     

\begin{remark}
          Theorem \ref{th:CLT} can also be proved using Malliavin-Stein's method for the CLT of SPDEs, originally introduced in \cite{HNV2018}. For instance, one can show that for fixed $t>0$, 
          there exists a constant $c>0$ depending on $t$ such that for all $N\geq\e$,
          \begin{align}\label{TVDeq}
		d_{\rm TV} \left(  \frac{ X_{N}(t)}
		{ \sqrt{{\rm Var}(X_{N}(t))}} ~,~
		{\rm N}(0,1)\right) \leq  c\,\sqrt{\frac{\log N}{ N}},
	\end{align}
	where $d_{\rm TV}$ denotes the total variation distance, and ${\rm N}(0,1)$
	denotes the standard normal distribution. The proof of \eqref{TVDeq} is similar to that of \cite[Theorem 1.2]{CKNPa} and the main difference is that we need the estimate on the second Malliavin derivative of the solution to \eqref{eq:SHE}, which was obtained  in \cite[Theorem 1.1]{KuN23}. We do not include the details here. 
\end{remark}

\begin{remark}
          It is natural to ask if a functional CLT holds as in Theorem \ref{th:CLT}. In the case of flat initial condition, the function CLT for the spatial average of the KPZ solution has been proved by Chen et al. \cite[Theorem 2.3]{CKNP23}, where they established the tightness in Proposition 5.3. Their method does not seem to allow us to  prove the tightness for the spatial average of the solution to the KPZ equation from narrow wedge initial data. One technical issue is that in our case, for any $t_1\neq t_2$, the distribution of $\bar\G(t_2,x;0,0)- \bar{\G}(t_1,x;0,0)$ depends on $x$. We leave this problem to interested   {readers}. 
\end{remark}

  {
\begin{remark}
           One may replace the logarithm by a general function $f\in C^2(\R_+)$ and consider the spatial decorrelation 
           between $f(\bar{\G}(t,x;0, 0))$ and $f(\bar{\G}(t,0;0,0))$ as $x\to \infty$ and the corresponding CLT.  Suppose we   assume that there exist $M,p>0$ such that
           \begin{align*}
           |f(z)| +|f'(z)| + |f''(z)| \leq M(z^p+z^{-p}), \quad \text{for all $z>0$},
           \end{align*}
           which would guarantee that
           \begin{align*}
           \EE[|g(\bar{\G}(t,0;0,0))|^k] \lesssim 1
           \end{align*}
           for any $g\in \{f, f', f''\}$ and $k\geq2$, then one can repeat the argument almost verbatim to obtain that for any fixed $t>0$
           \begin{align*}
           \Cov[f(\bar{\G}(t,x;0, 0))\,, f(\bar{\G}(t,0;0, 0))] \sim \frac{\sigma^2_{t,f}\,t}{x}, \quad \text{as $x\to\infty$}
           \end{align*}
           and as $N\to\infty$
           \begin{align*}
        \left\{\frac{1}{\sqrt{N\log N}}\int_0^N\left(f(\bar{\G}(t,x;0, 0))-\EE[f(\bar{\G}(t,x;0, 0))]\right)\d x\right\}_{t\in \R_+}\xrightarrow{\rm fdd} 
        \sqrt{2}\sigma_{\cdot,f}{\rm B},
       \end{align*}
       where $\sigma_{t,f}= \EE[\bar{\G}(t,0;0,0) f'(\bar{\G}(t,0;0,0))]$. In the case of $\sigma_{t,f}\neq0$,   the above result leads to the same phenomenon as in Theorems~\ref{th:cov} and \ref{th:CLT}.
       \end{remark}}


\begin{appendix}

\section{Technical estimates on integrals}\label{s.tech} 
This section is devoted to several technical lemmas concerning estimates of (deterministic) integrals.
\begin{lemma}\label{lem:s0}
       Fix $t_1, t_2>0$. Then
       \begin{align*}
        &\lim_{N\to\infty} \bigg[\frac{1}{N\log N} \int_{[0, N]^2}\d x_1\d x_2\int_0^{t_1\wedge t_2}\d s\int_\R \d y\,
        s^{1/4}\\
        &\qquad\qquad\qquad\qquad\qquad \qquad p_{s(t_1-s)/t_1}(y-\frac{s}{t_1}x_1)p_{s(t_2-s)/t_2}(y-\frac{s}{t_2}x_2)\bigg]=0.
       \end{align*}
\end{lemma}
\begin{proof}
      By the semigroup property of the heat kernel and a change of variable, we write the left-hand side as
      \begin{align*}
        &\frac{t_1t_2}{N\log N} \int_0^{N/t_1}\d x_1\int_0^{N/t_2}\d x_2\int_0^{t_1\wedge t_2}\frac{\d s}{s}
        s^{1/4}p_{\frac{2}{s}-\frac{1}{t_1}-\frac{1}{t_2}}(x_2-x_1)\\
        &\quad \leq \frac{t_1t_2}{N\log N} \int_{[0, N/(t_1\wedge t_2)]^2}\d x_1\d x_2\int_0^{t_1\wedge t_2}\frac{\d s}{s^{3/4}}
        p_{\frac{2}{s}-\frac{1}{t_1}-\frac{1}{t_2}}(x_2-x_1)\\
        &\quad = \frac{\pi^{-1}t_1t_2}{(t_1\wedge t_2)\log N} \int_0^{t_1\wedge t_2}\frac{\d s}{s^{3/4}}
         \int_{\R} \frac{1-\cos z}{z^2}\exp\left[-\left(\frac{1}{s}-\frac{1}{2t_1}-\frac{1}{2t_2}\right)\frac{z^2}{(t_1\wedge t_2)^2N^2}\right]\d z,
      \end{align*}
      where the equality holds by \cite[Lemma A.2]{CKNPa}. 
      We bound the exponential term above by $1$ so that the integral is bounded from above by $1/\log N$ multiple of a constant, which completes the proof. 
\end{proof}

\begin{lemma}\label{lem:2}
          Fix $t_1, t_2>0$. Then
          \begin{align*}
          \lim_{N\to\infty}\frac{1}{N\log N}\int_{[0, N]^2}\d x_1\d x_2\int^{(t_1\wedge t_2)/N^2}_{0}\frac{\d s}{s}
                 p_{(t_1-s)/(t_1s) + (t_2-s)/(t_2s)}(\frac{x_1}{t_1}- \frac{x_2}{t_2})=0.
          \end{align*}
\end{lemma}
\begin{proof}
          By change of variable and writing the integral in the Fourier domain, we write the term on the left-hand side as
          \begin{align*}
          &\frac{t_1t_2}{\log N}\int_0^{1/t_1}\d x_1\int_0^{1/t_2}\d x_2  \int^{(t_1\wedge t_2)/N^2}_{0}\frac{\d s}{s}p_{\frac2{sN^2}-(\frac1{t_1N^2}+\frac{1}{t_2N^2})}(x_1-x_2)\\
          &= \frac{t_1t_2}{2\pi\log N}  \int^{\frac{t_1\wedge t_2}{N^2}}_{0}\frac{\d s}{s} \int_\R \d z\,  \widehat{1_{[0, 1/t_1]}}(z) \overline{\widehat{1_{[0, 1/t_2]}}(z)}          \exp\left[-\left[\frac{1}{sN^2}- (\frac1{2t_1N^2}+\frac{1}{2t_2N^2})\right]z^2\right]\\
          & \leq  \frac{t_1t_2}{\log N}  \int^{\frac{t_1\wedge t_2}{N^2}}_{0}\frac{\d s}{s} \int_\R \d z\,  |\widehat{1_{[0, 1/t_1]}}(z) \overline{\widehat{1_{[0, 1/t_2]}}(z)}|
          \exp\left[-\left(\frac{1}{sN^2}- \frac1{(t_1\wedge t_2)N^2}\right)z^2\right]\\
          &=   \frac{t_1t_2}{\log N}  \int^{1}_{0}\frac{\d r}{r} \int_\R \d z\,  |\widehat{1_{[0, 1/t_1]}}(z) \overline{\widehat{1_{[0, 1/t_2]}}(z)}|
          \exp\left[-\left(\frac{1}{r}- \frac1{N^2}\right)z^2/(t_1\wedge t_2)\right],
          \end{align*}
          where the first equality holds by Parseval identity.  Notice that for all $N\geq1$ and $a>0$,
          \begin{align*}
          \int^{1}_{0}\exp\left[-\left(\frac{1}{r}- \frac1{N^2}\right)a\right] \frac{\d r}{r} &\leq \int^{1}_{0}\exp\left[-\left(\frac{1-r}{r}\right)a\right] \frac{\d r}{r}\\
          &=\e^a \int_a^\infty \frac{\e^{-s}\d s}{s} \leq  \e \log (\e+\frac\e a).
          \end{align*}
          Taking into account 
          \begin{align*}
          \left|\widehat{1_{[0, 1/t_1]}}(z) \overline{\widehat{1_{[0, 1/t_2]}}(z)} \right| \lesssim \frac{1\wedge z^2}{z^2} \quad \text{for all $z\neq 0$}
          \end{align*}
          and
          \begin{align*}
          \int_\R \frac{1\wedge z^2}{z^2}    \log (\e+\frac{\e (t_1\wedge t_2)}{z^2})\d z <\infty,
          \end{align*}
          the result follows. 
\end{proof}

\begin{lemma} \label{lem:y}
      Fix $t_1,t_2>0$. Then 
            \begin{align*}
      &\lim_{N\to\infty}\bigg[\frac{1}{N}
      \int_0^{N/t_1}\d x_1\int_0^{N/t_2}\d x_2  \int_0^2\d \tau\int_\R \d y\,\\
      &\qquad\qquad\qquad \qquad\qquad
        p_{\frac{N^{\tau}}{t_1\wedge t_2}-\frac{1}{t_1}}(y+x_2-x_1)
       p_{\frac{N^\tau}{t_1\wedge t_2}-\frac{1}{t_2}}(y)(1\wedge |N^{-\tau}y|^{1/2})\bigg]=0.
       \end{align*}
\end{lemma}
\begin{proof}
      We bound the integral $\int_0^{N/t_1}\d x_1$ above by $\int_\R \d x_1$ in order to see that left-hand side is bounded above by
      \begin{align*}
       &t_2^{-1}\int_0^2\d \tau\int_\R \d y\,
       p_{\frac{N^\tau}{t_1\wedge t_2}-\frac{1}{t_2}}(y)(1\wedge |N^{-\tau}y|^{1/2})\\
       &\qquad\qquad\qquad \qquad
       =t_2^{-1}\int_0^2  {\d} \tau\int_\R\d y\, p_{1}(y)(1\wedge |N^{-\tau}y\sqrt{\frac{N^\tau}{t_1\wedge t_2}-\frac{1}{t_2}}|^{1/2}),
      \end{align*}
      which converges to $0$ as $N\to\infty$ by dominated convergence theorem. 
\end{proof}

\begin{lemma}\label{lem:twotime}
       Fix $t_1,t_2>0$. Then
      \begin{align*}
        \lim_{N\to\infty} \frac{t_1t_2}{N}
      \int_0^{N/t_1}\d x_1\int_0^{N/t_2}\d x_2  \int_0^2\d \tau\,
        p_{\frac{2N^{\tau}}{t_1\wedge t_2}-\frac{1}{t_1}-\frac{1}{t_2}}(x_2-x_1)
        &= 2(t_1\wedge t_2).
       \end{align*}
\end{lemma}
\begin{proof}
       By Parseval's identity, we write the left-hand side as
       \begin{align*}
        &\frac{t_1t_2}{2\pi N} \int_0^2\d \tau \int_\R \widehat{1_{[0, N/t_1]}}(z)\overline{\widehat{1_{[0, N/t_2]}}}(z)\exp\left[-\left(\frac{N^{\tau}}{t_1\wedge t_2}-\frac{1}{2t_1}-\frac{1}{2t_2}\right)z^2\right]\d z\\
        &\quad =\frac{t_1t_2}{2\pi } \int_0^2\d \tau \int_\R 
        \frac{1-\e^{{\rm i}z/t_1}}{z {\rm i}}\frac{1-\e^{-{\rm i}z/t_2}}{-z {\rm i}}
        \exp\left[-\left(\frac{N^{\tau}}{t_1\wedge t_2}-\frac{1}{2t_1}-\frac{1}{2t_2}\right)\frac{z^2}{N^2}\right]\d z.
       \end{align*}
       Because
       \begin{align*}
        \left|\frac{1-\e^{{\rm i}z/t_1}}{z {\rm i}}\frac{1-\e^{-{\rm i}z/t_2}}{-z {\rm i}}\right| \lesssim \frac{1\wedge z^2}{z^2}, \quad \text{for all $z\neq 0$},
       \end{align*}
       by dominated convergence theorem, as $N\to\infty$ the preceding converges to 
       \begin{align*}
        \frac{t_1t_2}{2\pi } \int_0^2\d \tau \int_\R 
        \frac{1-\e^{{\rm i}z/t_1}}{z {\rm i}}\frac{1-\e^{-{\rm i}z/t_2}}{-z {\rm i}}\d z= 2t_1t_2\int_\R 1_{[0, 1/t_1]}(z)1_{[0, 1/t_2]}(z)\d z
        =2(t_1\wedge t_2).
       \end{align*}
    The proof is complete. 
\end{proof}

\end{appendix}


\begin{thebibliography}{99}

\bibitem{chris}
Alberts, T., Janjigian, C., Rassoul-Agha, F. and Sepp\"al\"ainen, T.: The Green's function of the parabolic Anderson model and the continuum directed polymer. arXiv preprint arXiv:2208.11255 (2022).

\bibitem{CDRP}
Alberts, T., Khanin, K. and  Quastel, J. (2014). The continuum directed random polymer. Journal of Statistical Physics, 154(1), 305-326.

\bibitem{ACQ11}
	Amir, G., Corwin, I. and Quastel, J. (2011).
	Probability distribution of the free energy of the continuum directed random polymer in $1+1$ dimensions.
	\textit{Comm. Pure Appl. Math.} \textbf{64}, 466--537.
    
\bibitem{BBF23}
Basu, R., Busani, O. and Ferrari, P.L. (2023).
On the exponent governing the correlation decay of the Airy$_1$ process. 
{\it Comm. Math. Phys.} {\bf 398} 1171--1211.


\bibitem{BNZ21}
Bolaños Guerrero, R., Nualart, D. and Zheng, G. (2021). Averaging 2d stochastic wave equation. {\it Electron. J. Probab.} {\bf 26},  pp.1-32.



\bibitem{CHN21}
Chen, L., Hu, Y. and Nualart, D. (2021).
Regularity and strict positivity of densities for the nonlinear stochastic heat equation.
{\it Mem. Amer. Math. Soc.} {\bf 273} no. 1340, v+102 pp.




\bibitem{CKNPa}
 Chen, L., Khoshnevisan, D., Nualart, D., and  Pu, F. (2022). Spatial ergodicity and central limit theorems for parabolic Anderson model with delta initial condition. {\it J. Funct. Anal.}, {\bf 282}(2), 109290.

 
\bibitem{CKNP23}
Chen, L., Khoshnevisan, D., Nualart, D., Pu, F. (2023).
Central limit theorems for spatial averages of the stochastic heat equation via Malliavin-Stein's method. 
{\it Stoch. Partial Differ. Equ. Anal. Comput. } {\bf 11} 122--176.

\bibitem{CKNP22}
Chen, L., Khoshnevisan, D., Nualart, D.,  Pu, F. (2022). Central limit theorems for parabolic stochastic partial differential equations. {\it Ann. Inst. Henri Poincaré Probab. Stat.} {\bf 58}, no. 2, 1052–1077

\bibitem{CKNP21}
Chen, L., Khoshnevisan, D., Nualart, D., Pu, F. (2021). Spatial ergodicity for SPDEs via Poincaré-type inequalities.{\it Electron. J. Probab.} {\bf 26}, 1-37.

\bibitem{CJ26}
Chen, L., and Jim\'enez, Juan. (2026). Spatial covariance of KPZ from flat initial profile. arXiv preprint arXiv:2603.14174 (2026).


\bibitem{icreview}
Corwin, I. (2012). The Kardar-Parisi-Zhang equation and universality class.
{\it Random Matrices Theory Appl.} {\bf 1} no. 1, 1130001, 76 pp.

%

\bibitem{landscape}
Dauvergne, D., Ortmann, J. and Vir\'ag, B. (2022)
The directed landscape.
{\it Acta Math.} {\bf 229} no. 2, 201–285.


\bibitem{DNZ20}
Delgado-Vences, F., Nualart, D., and Zheng, G. (2020). A central limit theorem for the stochastic wave equation with fractional noise. {\it Ann. Inst. Henri Poincar\'e Probab. Stat.}  {\bf 56}  no. 4, 3020–3042.3020-3042.

 \bibitem{hule}
Hu, Y. and  L\^e, K. (2022).
Asymptotics of the density of parabolic Anderson random fields.
{\it Ann. Inst. Henri Poincar\'e Probab. Stat.} {\bf 58}  no. 1, 105–133.

\bibitem{HNV2018}
	Huang, J., Nualart,D. and Viitasaari, L. (2020).
	A central limit theorem for the stochastic heat equation.
	{\it Stochastic Process. Appl.}\ {\bf 130}, no. 12, 7170--7184.
	
	\bibitem{HNVZ20}
	Huang, J., Nualart, D., Viitasaari, L., Zheng, G. (2020).Gaussian fluctuations for the stochastic heat equation with colored noise.
{\it Stoch. Partial Differ. Equ. Anal. Comput.} {\bf8} no. 2, 402–421.


\bibitem{KPZ86}
Kardar, M., Parisi, G. and Zhang, Y.-C. (1986). Dynamic scaling of growing interfaces. {\it Phys. Rev. Lett.} {\bf 56} 889.

\bibitem{KuN23}
Kuzgun, S. and Nualart, D.(2023).
Feynman-Kac formula for iterated derivatives of the parabolic Anderson model.
{\it Potential Anal.} {\bf 59} , no. 2, 651–673.

\bibitem{MQR21}
Matetski, K., Quastel, J. and Remenik, D. (2021).
The KPZ fixed point.
{\it Acta Math.} {\bf 227} 115--203.


 \bibitem{Nualart}
	Nualart, D. (2006).
	{\it  The Malliavin Calculus and Related Topics}.
	Springer, New York.
	
	\bibitem{NXZ22}
	Nualart, D., Xia, P.,  Zheng, G. (2022). Quantitative central limit theorems for the parabolic Anderson model driven by colored noises. {\it Electron. J. Probab.} {\bf 27}  Paper No. 120, 43 pp.

 

\bibitem{ps}
Prähofer, M.  and Spohn, H. (2002). Scale invariance of the PNG droplet and the Airy process. {\it J. Stat. Phys.}  {\bf 108} (2002): 1071-1106.


\bibitem{pufei}
Pu, F.: Ergodicity, CLT and asymptotic maximum of the Airy $ _1 $ process. arXiv e-prints (2023): arXiv-2311.



\bibitem{jqreview}
Quastel, J. (2012). 
Introduction to KPZ. Current developments in mathematics, 2011, 125–194.
International Press, Somerville, MA



\bibitem{QS15}
Quastel, J. and Spohn, H. (2015). The one-dimensional KPZ equation and its universality class.
{\it J. Stat. Phys.} {\bf 160} no. 4, 965–984.

\bibitem{stroock}
Stroock, D. W. (1987). {\it Homogeneous chaos revisited}. 
S\'eminaire de Probabilit\'es, XXI, 1–7.
Lecture Notes in Math., 1247


\bibitem{Vir20}
Vir\'ag, B. (2020). The heat and the landscape I. arXiv:2008.07241



\bibitem{wx}
Wu, Xuan. The KPZ equation and the directed landscape. arXiv preprint arXiv:2301.00547 (2023).



\end{thebibliography}
\end{document}